\documentclass[11pt, twoside, leqno, a4]{amsart}  
\usepackage{lipsum}
\usepackage{amsfonts}
\usepackage{graphicx}
\usepackage{epstopdf}
\usepackage{algorithmic}
\usepackage{calligra}
\usepackage[dvipsnames]{xcolor}
\usepackage{amsfonts,amsmath,amsthm,amssymb}
\usepackage{enumitem}
\usepackage{mathtools}
\usepackage{hyperref}
\usepackage{autonum}
\usepackage{hhline}
\usepackage{array}
\usepackage{diagbox}
\usepackage{tcolorbox}
\usepackage{mdframed}
\usepackage{multicol}
\usepackage{graphicx}
\usepackage{subcaption}
\usepackage{moreverb}
\usepackage{bbm}
\usepackage[margin=1.4in]{geometry}
\usepackage{todonotes}
\usepackage{scalerel,amssymb}
\allowdisplaybreaks
\usepackage{mathrsfs}  
\usepackage{lineno}
\usepackage{todonotes}
\usepackage{tikz}
\usepackage{accents}

\usepackage{pgfplots}
\pgfplotsset{compat=1.7}
\usetikzlibrary{arrows.meta}
\usepackage[numbers,sort&compress]{natbib}
\usepackage{appendix}
\DeclareMathAlphabet{\mathdutchcal}{U}{dutchcal}{m}{n}
\SetMathAlphabet{\mathdutchcal}{bold}{U}{dutchcal}{b}{n}
\DeclareMathAlphabet{\mathdutchbcal}{U}{dutchcal}{b}{n}

\newcommand{\dt}{\, \textup{d} t}
\newcommand{\ds}{\, \textup{d} s }
\newcommand{\dx}{\, \textup{d} x}

\newcommand{\Z}{\mathbb{Z}} 
\newcommand{\R}{\mathbb{R}} 
\newcommand{\supp}{\mathrm{Supp}}

\newcounter{rownumber}

\newcommand{\vecc}[1]{\boldsymbol{#1}}

\usepackage{scalerel,stackengine}
\stackMath
\newcommand\reallywidehat[1]{%
\savestack{\tmpbox}{\stretchto{%
  \scaleto{%
    \scalerel*[\widthof{\ensuremath{#1}}]{\kern-.6pt\bigwedge\kern-.6pt}%
    {\rule[-\textheight/2]{1ex}{\textheight}}
  }{\textheight}%
}{0.5ex}}%
\stackon[1pt]{#1}{\tmpbox}%
}
\newtheorem{theorem}{Theorem}
\newtheorem{lemma}{Lemma}
\newtheorem{proposition}{Proposition}
\newtheorem*{assumption*}{Assumptions on the memory kernel}
\newtheorem{definition}{Definition}
\newtheorem{remark}{Remark}
\newtheorem{corollary}{Corollary}
\numberwithin{corollary}{section}
\numberwithin{lemma}{section}
\numberwithin{proposition}{section}
\numberwithin{theorem}{section}
\numberwithin{equation}{section}

\makeatletter
\newcommand{\leqnomode}{\tagsleft@true}
\newcommand{\reqnomode}{\tagsleft@false}
\makeatother
       
\title[The JMGT equation in Besov spaces]{Global existence for the Jordan--Moore--Gibson--Thompson equation in Besov spaces}   


\author[B. Said-Houari]{\bfseries  Belkacem Said-Houari}

\address{  
	Department of Mathematics\\ College of Sciences\\ University of
Sharjah, P. O. Box: 27272 \\ Sharjah, United Arab Emirates}
\email{bhouari@sharjah.ac.ae}

\begin{document}  
\vspace*{-4mm}       
\maketitle           
\vspace*{-4mm}  
\begin{center}  
	{\footnotesize
	Department of Mathematics, University of Sharjah, United Arab Emirates    
	}
\end{center}	
\vspace*{3mm}
\begin{abstract}    
In this paper,  we consider the Cauchy problem of a model in nonlinear acoustic, named the Jordan--Moore--Gibson--Thompson equation.   This 
equation arises   as an alternative model to the well-known
Kuznetsov equation in acoustics.
We prove global existence and optimal time decay of solutions   in Besov spaces  with a minimal regularity assumption on the initial data, lowering  the regularity assumption required in \cite{Racke_Said_2019} for the proof of the global existence. Using a time-weighted energy method with the help of appropriate Lyapunov-type estimates,  we also extend  the decay rate in    \cite{Racke_Said_2019} and show an optimal decay rate of the solution  for initial data in the  Besov  space $\dot{{B}}_{2,\infty}^{-3/2}(\mathbb{R}^3)$, which is larger than the Lebesgue   space $L^1(\R^3)$ due to the embedding $L^1(\mathbb{R}%
^3)\hookrightarrow \dot{{B}}_{2,\infty}^{-3/2}(\mathbb{R}^3)$. Hence we removed the $L^1$-assumption on the initial data required in \cite{Racke_Said_2019} in order to prove the decay estimates of the solution.                           
\end{abstract}                
\vspace*{2mm}       
                
\section{Introduction}

Nonlinear wave propagation has been proved recently to be interesting in ultrasound imaging and in other medical applications, such as lithotripsy and thermotherapy. See for instance \cite{maresca2017nonlinear, he2018shared, melchor2019damage, duck2002nonlinear}. One of the classical models in nonlinear acoustics is the Kuznetsov equation     
\begin{equation}\label{Kuznt}
     	u_{tt}-c^{2}\Delta u-\nu\Delta u_{t}=\dfrac{\partial}{\partial t}\left( \dfrac{1}{c^{2}}\dfrac{B}{2A}(u_{t})^{2}+|\nabla u|^{2}\right), 
     	\end{equation}
	where $v(x,t)=-\nabla u(x,t)$ represents the acoustic velocity potential.  
	Equation \eqref{Kuznt} was initially introduced in \cite{kuznetso_1971} and  derived from a compressible nonlinear isentropic Navier--Stokes (for $\nu>0$) and Euler (for $\nu=0$) systems by assuming in the equation of the conservation of energy  that the heat flux obeys the   Fourier law of heat conduction:  
	\begin{equation}\label{Fourier}
     	\vecc{q}(x,t)=-K\nabla\theta(x,t),
     	\end{equation}
	where $K$ is the thermal conductivity of the material and $\theta$ is the absolute temperature. 
The interested reader is referred to the recent paper by 	Dekkers and Rozanova-Pierrat \cite{Dekkers_Rozanova-Pierrat_2019} where the authors investigated \eqref{Kuznt}  and proved a global existence result in the viscous case form small initial data of size roughly $\varepsilon \nu^{1/2}$, where  $\varepsilon$ is a small parameter.  On the other hand, for the inviscid case, they proved a blow-up result.

	It is well-known that using the Fourier law \eqref{Fourier} may lead to the  infinite signal speed paradox of the energy propagation; see \cite{Chand86,JLP89}. To overcome  this drawback  in the Fourier law, other equations were considered to model the heat transfer.  The Cattaneo (or the Maxwell--Cattaneo) law:  
     	\begin{equation}\label{Cattaneo}
     	\tau \vecc{q}_{t}(x,t)+\vecc{q}(x,t)=-K\nabla\theta(x,t),
     	\end{equation}
	leads to a finite speed of propagation. In \eqref{Cattaneo}
     	  $\tau$ is a small parameter known as  the relaxation time of the heat flux.   The use  of the Cattaneo law instead of the Fourier law in the governing equations of fluid dynamics leads to  the following third order ``in time" equation know as the Jordan--Moore--Gibson--Thompson equation: 
\begin{equation} \label{Main_problem} 
\left.
\begin{array}{ll}
\tau u_{ttt}+ u_{tt}-c^{2}\Delta u-\beta \Delta u_{t}=\dfrac{\partial }{%
\partial t}\left( \dfrac{1}{c^{2}}\dfrac{B}{2A}(u_{t})^{2}+|\nabla
u|^{2}\right) ,\vspace{0.2cm} \\
u(t=0)=u_{0},\qquad u_{t}(t=0)=u_{1}\qquad u_{tt}(t=0)=u_{2},
\end{array}
\right. 
\end{equation}%
where $x\in \Omega $ ($\Omega$ is either a bounded domain of $\R^n$ or $\Omega=\R^n$)  and $t>0$
, and where $\tau>0$ is a time relaxation parameter,  $c$ is the speed of sound,  $\beta$ is the parameter of diffusivity and $A$ and $B$ are the constants of nonlinearity. The reader is referred to the papers   \cite{jordan2008nonlinear,Kaltenbacher_2011} for the derivation of \eqref{Main_problem}. We also mention the recent result \cite{Kalten_Niko_2021}, where different time-fractional (J)MGT models have been derived by considering instead of \eqref{Cattaneo} a time-fractional heat-flux law.   

In recent years there has been a lot of work investigating the JMGT equation and its linearized version; the Moore--Gibson--Thompson (MGT) equation.    
The study of the controllability properties of the MGT type equations can be found, for instance in \cite{bucci2019feedback, Lizama_Zamorano_2019}.
The MGT equation in $\R^n$ with a power source nonlinearity of the form $|u|^p$  has been considered in \cite{Chen_Palmieri_1}  where some blow up results have been shown for the critical case $\tau c^2=\beta$.
The MGT and JMGT equations with a memory term have been also investigated recently. For the  MGT with memory, the reader is referred to  \cite{Bounadja_Said_2019,Liuetal._2019,dell2016moore} and to \cite{lasiecka2017global,nikolic2020mathematical,Nikolic_SaidHouari_2} for the JMGT with memory.  
The singular limit problem when $\tau\rightarrow 0$ has been rigorously justified  in \cite{KaltenbacherNikolic}. The authors in \cite{KaltenbacherNikolic} showed that in bounded domains,    the limit of \eqref{Main_problem}  as $\tau \rightarrow 0$ leads to the Kuznetsov equation \eqref{Kuznt}. 
    
For the nonlinear model \eqref{Main_problem} in a bounded domain, the first result seems to be the one of Kaltenbacher et \emph{al.} in \cite{KatLasPos_2012}. They investigated the equation \eqref{Main_problem} in its  abstract form and proved that the problem is well-posed for the linear equation with variable viscosity and positive diffusivity and showed the exponential decay of solutions of the linear   equation. After that they gave local and global well-posedness and exponential decay for the nonlinear equation in a certain range of the parameters and for small initial  data.

The  Cauchy problem associated to the MGT equation was first  studied by  Pellicer  and the author of this paper in \cite{PellSaid_2019_1}, where we showed the well-posedness and investigated the decay rate of the solution    in the whole space $\R^{n}$.  We used the energy method in the Fourier space to show that under the assumption $\beta>c^{2}\tau$ the $L^{2}$-norm of the vector  $\mathbf{V}=(u_{t}+\tau u_{tt},\nabla(u+\tau u_{t}),\nabla u_{t}),$ and of its higher-order derivatives decay as 
        \begin{align}\label{marta decay}
        \Vert\nabla^{j}\mathbf{V}(t)\Vert _{L^{2}(\R^{n})}\leq C(1+t)^{-{n}/{4-{j}/{2}}}\Vert \mathbf{V}_{0}\Vert _{L^{1}(\R^{n})} +Ce^{-ct} \Vert\nabla^{j}\mathbf{V}_{0}\Vert _{L^{2}(\R^n)},
        \end{align} 
        with $\mathbf{V}_0=\mathbf{V}(t=0)\in H^s(\R^n)\cap L^1(\R^n)$ and $0\leq j\leq s$.  In addition, we proved the optimality of the decay rate in  \eqref{marta decay} by the eigenvalues expansion method. We also showed  the following asymptotic behavior of the $L^2$-norm of the solution: 
\begin{equation}\label{Decay_rate_u}
\begin{aligned}
\Vert u(t)\Vert_{L^2} \lesssim&\, (1+t)^{\frac{1}{2}-\frac{n}{4}},\qquad &n\geq 3\\
\Vert u(t)\Vert_{L^2} \lesssim&\, (1+t)^{1/2},\qquad &n=2. 
 \end{aligned}
\end{equation}
It is clear from \eqref{Decay_rate_u} that the $L^2$-norm of the solution either decays very slowly (for $n\geq 3$) or has a time growth (for $n=2$). This creates  a  difficulty in the nonlinear problem, since the integrability with respect to $t$ of the $L^2$-norm of the solution will be lost.  
         See also the recent paper by Chen and Ikehata \cite{Chen_Ikehata_2021}, where the authors proved the  optimal growth or optimal decay of the norm $\Vert u(t)\Vert_{L^2}$, depending on the space dimension.    
        

 In \cite{Racke_Said_2019},  the author together with Racke  considered the JMGT problem \eqref{Main_problem} in the whole space  $\R^{3}$ (the 3D case) and showed a local existence result in appropriate function spaces by using the contraction mapping theorem and a global existence result for small data, by using the energy method together with some interpolation inequalities such as the classical Gagliardo--Nirenberg interpolation inequality. 
       We also proved  certain  decay rates for the solution for small initial data in $H^s(\R^3)\cap L^1(\R ^3) $, for $s>\frac{5}{2}=\frac{n}{2}+1.$ The proof of the global existence in \cite{Racke_Said_2019} does not use the linear decay of the solution. It is based on nonlinear energy estimates. 

Our goal in this paper is to consider system \eqref{Main_problem}  in the three dimensional case  
and  improve the global existence result and decay estimates in \cite{Racke_Said_2019} by reducing  the regularity  assumption on the initial data. In fact, we show  that for small initial data in some appropriate  Besov spaces,   system \eqref{Main_problem} has a unique global solution. These spaces considered here are larger than the Sobolev spaces in \cite{Racke_Said_2019} (see the discussion in Section \ref{Sec:Discussion}).  This of course affects the size of the initial data required for the global existence, since the larger the space is, the smaller the initial data are.  For instance,  in \cite[page 30]{bahouri2011fourier}  the authors  considered:    
$\phi_\varepsilon=e^{i\frac{x_1}{\varepsilon}}\phi(x)$ for some function $\phi$ in $\mathcal{S}(\R^n)$ and  showed  that $\Vert\phi_\varepsilon\Vert_{\dot{H}^s}\approx \varepsilon^{-s} $, $\Vert\phi_\varepsilon\Vert_{L^p} \approx 1$ and $\Vert\phi_\varepsilon\Vert_{\dot{B}^{-\sigma}_{p,q}(\R^n)}\approx \varepsilon^{\sigma}$. It is clear that $\dot{H}^{\frac{n}{2}-1}(\R^n) \hookrightarrow L^n(\R^n) \hookrightarrow \dot{B}^{-1+\frac{n}{p}}_{p,r}(\R^n),\, r\geq 2 $.  Hence,  for $\varepsilon$ small, $\phi_\varepsilon$ has a large norm in $\dot{H}^{\frac{n}{2}-1}(\R^n)$ but a small norm in the larger space $\dot{B}^{-1+\frac{n}{p}}_{p,r}(\R^n)$ if $ p>d$. For this reason it is always interesting to show small-data global existence in the  largest possible space. \\     
To prove the global existence result, we use  nonlinear 
energy-type estimates  by constructing an appropriate energy norm in some Besov spaces with critical regularity  and show that this norm
remains uniformly bounded with respect to time. We point out that in the proof of global existence result, we  
 do  not  use the linear decay,  which is a standard way of   proving small data existence for non-linear evolution equations. Our proof here is purely based on energy method.  
 In order  to prove  the decay rate for dissipative partial differential equations, it is quite  common to make smallness assumption on the $L^1$-norm of the initial data and combine it with the $L^2$-type energy estimates through Duhamel principle  to obtain large time decay estimates.   However, in many situations, it is difficult to work with the $L^1$-norm since it is not always possible to propagate the $L^1$-norm along the time evolution. So, it is of a great interest to prove decay estimates for initial data in $L^2$ based spaces that contain $L^1$.   In this paper and inspired by \cite{Sohinger_Strain_2014}, we prove  decay estimates for the solution of \eqref{Main_problem} for initial data in the homogeneous Besov spaces $\dot{{B}}_{2,\infty}^{-3/2}(\mathbb{R}^3)$, which contains $L^1(\R^3)$ (see Lemma \ref{Embedding_Lemma} below). 

 The rest of this paper is organized as follows:  Section~\ref{section_prel} contains the necessary theoretical preliminaries, which allow us to rewrite the equation with the corresponding initial data as a  first-order Cauchy problem and define the main energy norm with the associated dissipative norm in appropriate Besov-type spaces.      In Section~\ref{Sect_Main_Result}, we state and discuss  our main result. In Section~\ref{Section_Global_Existence}, we derive the main energy estimate and  prove  the global existence result.  
 Section~\ref{Decay_Estimate_Linearized_Model} is dedicated to the proof of the decay estimates of the linearized problem, where we extend and improve the result in \cite{PellSaid_2019_1}.  In Section \ref{Section_6}, we prove the decay estimates of the nonlinear problem. 
 In Appendix \ref{Appendix_A}, we recall  the Littlewood--Paley decomposition theory and give the definition and some useful properties of the Besov spaces. 
 We also present a useful interpolation inequality in Besov spaces and other integral inequalities that  we used in the proofs.

\subsection{Notation} Throughout the paper, the constant $C$ denotes a generic positive constant
that does not depend on time, and can have different values on different occasions.
We often write $f \lesssim g$ where there exists a constant $C>0$, independent of parameters of interest such that   $f\leq C g$ and we analogously define $f\gtrsim g$. We sometimes use the notation $f\lesssim_\alpha g$ if we want to emphasize that the implicit constant depends on some parameter $\alpha$. The notation $f\approx g$ is used when there exists a constant $C>0$ such that $C^{-1}g\leq f\leq Cg$. We  define the operator $\Lambda^\gamma$ for $\gamma\in \R $ by 
\begin{equation}
\Lambda^\gamma f(x)=\int_{\R^3} |\xi|^\gamma \hat{f}(\xi) e^{2i\pi x\cdot \xi} \, \textup{d}\xi,
\end{equation}
where $\hat{f}$ is the Fourier transform of $f$.

\section{Preliminaries}

\label{section_prel} In this section, we  define the main energy norm with the associated dissipative norm in appropriate Besov-type spaces.  
 To do this and to state our main result on the global existence and asymptotic decay, we first rewrite equation \eqref{Main_problem} as a first-order in time system,
  which is more convenient for analysis. Hence, we first  introduce the new variables
\begin{equation}\label{Change_Variables}
v=u_{t}\qquad \text{ and }\qquad w=u_{tt},
\end{equation}%
and rewrite the right-hand side of the first equation in \eqref{Main_problem} in the form
\begin{equation}\label{MGT_1_2}
\frac{\partial }{\partial t}\left( \frac{1}{c^{2}}\frac{B}{2A}%
(u_{t})^{2}+|\nabla u|^{2}\right) =\frac{1}{c^{2}}\frac{B}{A}%
u_{t}u_{tt}+2\nabla u\nabla u_{t}. 
\end{equation}
Thus, taking into account \eqref{MGT_1_2} and \eqref{Change_Variables}, we recast \eqref{Main_problem} as (without loss of generality, we assume from now on $c=1$):
  \begin{subequations}\label{M_Main_System}
\begin{equation}
\left\{
\begin{array}{ll}
u_{t}=v,\vspace{0.1cm} &  \\
v_{t}=w,\vspace{0.1cm} &  \\
\tau w_{t}=\Delta u+\beta \Delta v-w+\dfrac{B}{A}vw+2\nabla u\nabla v , &
\end{array}%
\right.  \label{System_New}
\end{equation}
with the initial data
\begin{eqnarray}  \label{Initial_Condition_2}
u(t=0)=u_0,\qquad v(t=0)=v_0,\qquad w(t=0)=w_0.
\end{eqnarray}
\end{subequations}
Let $\mathbf{U}=(u,v,w)^T$ be the solution of \eqref{M_Main_System}. 
In order to state our main result, we introduce the energy norm, $\mathcal{E}[\mathbf{U}]
(t)$ and the corresponding dissipation norm $\mathcal{D}[\mathbf{U}](t)$, as
follows:%

\begin{equation}\label{Weighted_Energy}
\begin{aligned}
\mathcal{E}^2[\mathbf{U}](t):=&\,\sup_{0\leq \sigma\leq t}\Vert \mathbf{U}(\sigma)\Vert_{\mathcal{L}^2}^2
\end{aligned}
\end{equation}
where 
\begin{equation}
\begin{aligned}
\Vert \mathbf{U}(t)\Vert_{\mathcal{L}^2}^2:=&\,\big\Vert (v+\tau
w)(t)\big\Vert _{L^{2}}^{2}+\big\Vert \nabla(v+\tau
w)(t)\big\Vert _{L^{2}}^{2}+\big\Vert \Delta  v(t)\big\Vert
_{L^{2}}^{2}\Big.%
\vspace{0.2cm}  \notag \\
 \Big.&+\big\Vert \nabla v(t)\big\Vert _{L^{2}}^{2}+\big\Vert \Delta (u+\tau v)(t)\big\Vert
_{L^{2}}^{2}+\big\Vert \nabla (u+\tau v)(t)\big\Vert
_{L^{2}}^{2}+\Vert w(t)\Vert_{L^2}^2.
\end{aligned}
\end{equation}
We also we define 
\begin{equation}\label{Dissipative_weighted_norm_1}
\mathcal{D}^{2}[\mathbf{U}](t) :=\int_0^t\Vert \mathbf{U}(t)\Vert_{\mathbb{L}^2}^2(\sigma)\textup{d}\sigma 
\end{equation}
with 
\begin{equation}
\begin{aligned}
\Vert \mathbf{U}(t)\Vert_{\mathbb{L}^2}^2:=&\,\big\Vert \nabla
v(t)\big\Vert _{L^{2}}^{2}+\big\Vert \Delta 
v(t)\big\Vert
_{L^{2}}^{2}+ \Vert w (t)\Vert_{L^2}^2   \\
&+\big\Vert \Delta \left( u+\tau v\right)(t) \big\Vert
_{L^{2}}^{2} + \big\Vert \nabla (v+\tau w)(t)\big\Vert _{L^{2}}^{2}.
\end{aligned}
\end{equation}
We also introduce 
\begin{equation}
\begin{aligned}
M_0(t):=\sup_{0\leq \sigma\leq t}&\Big(\left\Vert v(s)\right\Vert _{L^{\infty
}}+\left\Vert (v+\tau w)(\sigma)\right\Vert _{L^{\infty }}\Big.\\
\Big.+&\left\Vert \nabla
(u+\tau v)(\sigma)\right\Vert _{L^{\infty }}+\Vert \nabla u(\sigma)\Vert_{L^\infty}+\Vert \nabla v(\sigma)\Vert_{L^\infty}\Big).
\end{aligned}
\end{equation}
We define (see Appendix \ref{Appendix_A} for the definition and properties of Besov spaces)
\begin{equation}\label{M_Functional}
\begin{aligned}
M(t):=&\,\Vert \nabla^2 u\Vert_{\widetilde{L}_t^\infty(\dot{B}^{1/2}_{2,1})}+\Vert \nabla^2 u\Vert_{\widetilde{L}_t^\infty(\dot{B}^{3/2}_{2,1})}+\Vert v\Vert_{\widetilde{L}_t^\infty(\dot{B}^{1/2}_{2,1})}+\Vert \nabla v\Vert_{\widetilde{L}_t^\infty(\dot{B}^{1/2}_{2,1})}\\
&+\Vert \nabla^2 v\Vert_{\widetilde{L}_t^\infty(\dot{B}^{1/2}_{2,1})}+\Vert w\Vert_{\widetilde{L}_t^{\infty}(\dot{B}^{1/2}_{2,1})}+\Vert \nabla w\Vert_{\widetilde{L}_t^{\infty}(\dot{B}^{1/2}_{2,1})}. 
\end{aligned}
\end{equation}
We also define for $s\in \R$ the following Besov-type norms 
\begin{equation}
\begin{aligned}
\Vert \mathbf{U}(t)\Vert_{\dot{\mathbf{\mathcal{B}}}_{2,1}^{s}}^2:=&\,\big\Vert (v+\tau
w)(t)\big\Vert _{\dot{B}_{2,1}^s}^{2}+\big\Vert \nabla(v+\tau
w)(t)\big\Vert _{\dot{B}_{2,1}^s}^{2}+\big\Vert \Delta  v(t)\big\Vert
_{\dot{B}_{2,1}^s}^{2}
 \\
&+\big\Vert \nabla v(t)\big\Vert _{\dot{B}_{2,1}^s}^{2}+\big\Vert \Delta (u+\tau v)(t)\big\Vert
_{\dot{B}_{2,1}^s}^{2}+\big\Vert \nabla (u+\tau v)(t)\big\Vert
_{\dot{B}_{2,1}^s}^{2}+\Vert w(t)\Vert_{\dot{B}_{2,1}^s}^2,
\end{aligned}
\end{equation}
and the associated dissipation norm 
\begin{equation}
\begin{aligned}
\Vert \mathbf{U}(t)\Vert_{\dot{\mathbb{B}}_{2,1}^{s}}^2:=&\,\big\Vert \nabla
v(t)\big\Vert _{\dot{B}_{2,1}^{s}}^{2}+\big\Vert \Delta 
v(t)\big\Vert
_{\dot{B}_{2,1}^{s}}^{2}+ \Vert w (t)\Vert_{\dot{B}_{2,1}^{s}}^2   \\
&+\big\Vert \Delta \left( u+\tau v\right)(t) \big\Vert
_{\dot{B}_{2,1}^{s}}^{2} + \big\Vert \nabla (v+\tau w)(t)\big\Vert _{\dot{B}_{2,1}^{s}}^{2}.
\end{aligned}
\end{equation}
Finally, we define for $s>0$,  the inhomogeneous Besov-type norms as:  
\begin{equation}
\begin{aligned}
\Vert \mathbf{U}(t)\Vert_{\mathbf{\mathcal{B}}_{2,1}^{s}}:=&\,\Vert \mathbf{U}(t)\Vert_{\mathcal{L}^2}+\Vert \mathbf{U}(t)\Vert_{\dot{\mathbf{\mathcal{B}}}_{2,1}^{s}}\\
\Vert \mathbf{U}(t)\Vert_{\mathbb{B}_{2,1}^{s}}:=&\,\Vert \mathbf{U}(t)\Vert_{\mathbb{L}^2}+\Vert \mathbf{U}(t)\Vert_{\dot{\mathbb{B}}_{2,1}^{s}}. 
\end{aligned}
\end{equation}
The above norms are carefully designed and take advantages of the  good behavior of some linear combinations of the components of the vector  solution $\mathbf{U}=(u,v,w)^T$. These linear combinations are very important and they help to find some cancelation properties in system \eqref{M_Main_System}, which allow us to  exploit   the dissipation nature of system \eqref{M_Main_System}. Observe that the control of the norm $\Vert \mathbf{U}(t)\Vert_{\mathcal{L}^2} $ does not offer any control on  the norm $\Vert u\Vert_{L^2}$ due to the absence of the  Poincar\'e's inequality in the whole space $\R^n$. Hence, in the light of \eqref{Decay_rate_u}, this makes the proof of the nonlinear energy estimates more challenging since we need to avoid the use of the decay estimates of the linearized problem and as a consequence we do not assume the $L^1$ bound on the initial data.

We also define the following norms:
\begin{equation}
\begin{aligned}
\Vert \mathbf{ U}(t)\Vert_{\widetilde{L}_T^\infty(\mathbf{\mathcal{B}}_{2,1}^{s})}:=&\,\big\Vert (v+\tau
w)(t)\big\Vert _{\widetilde{L}_T^\infty(B_{2,1}^s)}+\big\Vert \nabla(v+\tau
w)(t)\big\Vert _{\widetilde{L}_T^\infty(B_{2,1}^s)}+\big\Vert \Delta  v(t)\big\Vert
_{\widetilde{L}_T^\infty(B_{2,1}^s)}
 \\
&+\big\Vert \nabla v(t)\big\Vert _{\widetilde{L}_T^\infty(B_{2,1}^s)}+\big\Vert \Delta (u+\tau v)(t)\big\Vert
_{\widetilde{L}_T^\infty(B_{2,1}^s)}+\big\Vert \nabla (u+\tau v)(t)\big\Vert
_{\widetilde{L}_T^\infty(B_{2,1}^s)}\\
&+\Vert w(t)\Vert_{\widetilde{L}_T^\infty(B_{2,1}^s)}
\end{aligned}
\end{equation}  
and 
\begin{equation}
\begin{aligned}
\Vert \mathbf{ U}(t)\Vert_{\widetilde{L}^2_T(\mathbb{B}_{2,1}^{s})}:=&\,\big\Vert \nabla
v(t)\big\Vert _{\widetilde{L}^2_T(B_{2,1}^{s})}+\big\Vert \Delta 
v(t)\big\Vert
_{\widetilde{L}^2_T(B_{2,1}^{s})}+ \Vert w (t)\Vert_{\widetilde{L}^2_T(B_{2,1}^{s})}   \\
&+\big\Vert \Delta \left( u+\tau v\right)(t) \big\Vert
_{\widetilde{L}^2_T(B_{2,1}^{s})} + \big\Vert \nabla (v+\tau w)(t)\big\Vert _{\widetilde{L}^2_T(B_{2,1}^{s})}.
\end{aligned}
\end{equation}
We define the spaces  $\widetilde{\mathcal{C}}(\mathbf{\mathcal{B}}_{2,1}^{s})$ and $\widetilde{\mathcal{C}}^1(\mathbf{\mathcal{B}}_{2,1}^{s})$  as in \eqref{C_1_T_Def} and \eqref{C_1_infty_Def}, where $B_{2,1}^{s}$ is replaced by $\mathcal{B}_{2,1}^{s}$.

\section{Main results}\label{Sect_Main_Result}
\noindent In this section, we state and discuss  our main results. The global existence result is stated in Theorem \ref{Main_Theorem_Nonl}, while the decay estimates are  given in Theorem \ref{Decay_Estimate_Theorem}. (We refer to Appendix \ref{Appendix_A} for the definition of Besov spaces that we use in the statement of the main theorems below).

\begin{theorem}
\label{Main_Theorem_Nonl} Assume that $0<\tau<\beta$.
   Assume that $%
u_{0},v_0\in B^{7/2}_{2,1}(\mathbb{R}^3) $ and $w_0\in B^{5/2}_{2,1}(\mathbb{R}^3)$. Then there exists a small positive
constant $\delta$ such that if
\begin{eqnarray*}
\Vert\mathbf{ U}_0\Vert_{\mathbf{\mathcal{B}}_{2,1}^{3/2}}\leq \delta,
\label{Initial_Assumption_Samll}
\end{eqnarray*}
 problem  \eqref{Main_problem} has a unique global solution  satisfying 
\begin{equation}
\mathbf{ U}\in \widetilde{\mathcal{C}}(\mathbf{\mathcal{B}}_{2,1}^{3/2}(\R^3))\cap \widetilde{\mathcal{C}}^1(\mathbf{\mathcal{B}}_{2,1}^{1/2}(\R^3)). 
\end{equation}
In addition, the following energy inequality holds: 
\begin{equation}\label{Main_Estimate_Theorem}
\begin{aligned}
&\Vert \mathbf{ U}(t)\Vert_{\widetilde{L}_T^\infty(\mathbf{\mathcal{B}}_{2,1}^{3/2})}+\Vert \mathbf{ U}(t)\Vert_{\widetilde{L}^2_T(\mathbb{B}_{2,1}^{3/2})}
\lesssim\Vert \mathbf{ U}_0\Vert_{\mathbf{\mathcal{B}}_{2,1}^{3/2}}. 
\end{aligned}
\end{equation}
 
\end{theorem}

\begin{theorem}\label{Decay_Estimate_Theorem}
Let $\mathbf{U}(t,x)$ be the global solution given in Theorem \ref{Main_Theorem_Nonl}. Assume that $\mathbf{U}_0\in \mathbf{\mathcal{B}}_{2,1}^{3/2} $ and $\mathbf{V}_0=(u_1+\tau u_{2}, \nabla (u_0+\tau u_1), \nabla u_1)\in  \dot{B}_{2,\infty}^{-3/2}(\R^3) $.
Let 
\begin{equation}\label{E_0_ss}
\mathbbmss{E}_0=\Vert \mathbf{ U}_0\Vert_{\mathbf{\mathcal{B}}_{2,1}^{3/2}(\R^3)}+\Vert \mathbf{V}_0\Vert_ {\dot{B}_{2,\infty}^{-3/2}(\R^3)} 
\end{equation}
be sufficiently small. 
 Then, it holds that 
\begin{equation}
\Vert\Lambda^{\ell}\mathbf{V}(t)\Vert_{\mathbf{X}_1(\R^3)}\lesssim  \mathbbmss{E}_0 (1+t)^{-3/4-\ell/2 }
\end{equation}
for $0\leq \ell< 3/2 $, where $\mathbf{X}_1(\R^3)=B^{3/2-\ell}_{2,1}(\R^3)$ and $\mathbf{X}_1(\R^3)=\dot{B}^{0}_{2,1}(\R^3)$ if $\ell=3/2$. 

\noindent In addition the following decay estimates also hold:
\begin{equation}\label{Decay_w}
\Vert w(t)\Vert_{\dot{B}^{3/2}_{2,1}}\lesssim \mathbbmss{E}_0 (1+t)^{-3/2},\qquad \Vert w(t)\Vert_{\dot{B}^{0}_{2,1}}\lesssim \mathbbmss{E}_0 (1+t)^{-3/4}.
\end{equation}
 
\end{theorem}
As we will see the proof of Theorems \ref{Main_Theorem_Nonl} relies on nonlinear energy estimates, whereas to prove Theorem \ref{Decay_Estimate_Theorem}, we prove first a decay estimate for the linearized problem and then represent the solution in an integral from via the frequency localization  Duhamel principle (see \cite{Xu_Kawashima_2015} and \cite{Kawashima_1} for similar ideas).  
The proof of Theorem \ref{Main_Theorem_Nonl} is a consequence of Proposition \ref{Main_Proposition} and the bootstrap argument. The proof of Theorem \ref{Main_Theorem_Nonl} is given in Section \ref{Proof_Theorem_1} while the one of Theorem  \ref{Decay_Estimate_Theorem} is given in Section \ref{Sec:Proof_Theorem_2}. The proof of a local well-posedness result can be done by using similar methods  as in \cite[Theorem 1.2]{Racke_Said_2019}; i.e., by combining a fixed point argument together with the appropriate a priori estimates derived in Section \ref{Section_Global_Existence}.  We omit it here.   
\subsection{Discussion of the main result}\label{Sec:Discussion}
 Before moving onto the proof, we briefly discuss the statements made above in Theorems \ref{Main_Theorem_Nonl} and \ref{Decay_Estimate_Theorem}.

\begin{enumerate}[label=\arabic*.,ref=.\arabic*]
\item  The global existence result in \cite{Racke_Said_2019} requires $u_0, v_0\in H^{\frac{7}{2}+}$ and $w_0\in  H^{\frac{5}{2}+}$. Since $H^{s_0+}=B_{2,2}^{s_0+}\hookrightarrow B_{2,1}^{s_0}$ (see \cite[Proposition 2.3]{Yoshihiro_2018}),  the global existence result stated in Theorem \ref{Main_Theorem_Nonl}   improves \cite{Racke_Said_2019} for the minimal possible regularity. As we said in the introduction, this means  initial data which are large in the small space $H^{s_0+}$ might be  small in $B_{2,1}^{s_0+}$.  (Here $H^{s+}=H^{\delta}, B_{p,q}^{s+}=B_{p,q}^{\delta}$ for any  $\delta>s$). 
 In the proof of the main result, we took advantages of the embedding $\dot{B}^{n/p}_{p,1}(\R^n) \hookrightarrow L^\infty(\R^n)$. Such embedding fails if we replace the Besov space $\dot{B}^{n/p}_{p,1}(\R^n)$ by the homogeneous Sobolev space $\dot{H}^{n/p}(\R^n)$. 
\item By using the interpolation:
\begin{equation}
\Vert w(t)\Vert_{\dot{B}^{\sigma}_{2,1}} \lesssim \Vert w(t)\Vert_{\dot{B}^{3/2}_{2,1}}^{\frac{2}{3}\sigma} \Vert w(t)\Vert_{\dot{B}^{0}_{2,1}}^{1-\frac{2}{3}\sigma},
\end{equation}
 we have from \eqref{Decay_w} 
 \begin{equation}
\Vert w(t)\Vert_{\dot{B}^{\sigma}_{2,1}}\lesssim \mathbbmss{E}_0 (1+t)^{-\frac{\sigma}{2}-\frac{3}{4}}. 
\end{equation}
for $ 0\leq \sigma\leq  3/2$.
\item The decay estimates obtained in \cite{Racke_Said_2019} hold under the  assumption that the initial data $\mathbf{V}_0\in L^1(\R^3)$. Theorem~\ref{Decay_Estimate_Theorem} does not require the initial data to be in $L^1(\R^3)$. In fact, we assume that the initial data $\mathbf{V}_0$ to be in the Besov space $ \dot{B}_{2,\infty}^{-3/2}(\R^3) $ which satisfies $L^1(\mathbb{R}%
^3)\hookrightarrow \dot{{B}}_{2,\infty}^{-3/2}(\mathbb{R}^3)$.
\item In this paper, although we restrict ourselves to the  3D case,     the proof of the main results can be modified to include the general spaces dimension $\R^n,\, n\geq 3$. The case $n=2$ would require  new estimates. 
\item In Theorem \ref{Main_Theorem_Nonl},  since we do not use the decay estimates of the linearized problem, we do not assume the extra $ \mathbf{V}_0\in\dot{B}_{2,\infty}^{-3/2}(\R^3)$ for the proof of the global existence result.         
\end{enumerate}

\section{Energy estimates}\label{Section_Global_Existence}

The following proposition is crucial in the proof of Theorem \ref{Main_Theorem_Nonl}.  
\begin{proposition}\label{Main_Proposition}
Assume that $0<\tau<\beta$. Then  the following estimate holds: 
\begin{equation}\label{Main_inHomogenous_Estimate_1}
\begin{aligned}
&\Vert \mathbf{U}(t)\Vert_{\widetilde{L}_t^\infty(\mathbf{\mathcal{B}}_{2,1}^{3/2})}+\Vert \mathbf{U}(t)\Vert_{\widetilde{L}^2_t(\mathbb{B}_{2,1}^{3/2})}\\
\lesssim&\,\Vert \mathbf{U}(0)\Vert_{\mathbf{\mathcal{B}}_{2,1}^{3/2}}+\sqrt{\Vert \mathbf{U}(t)\Vert_{\widetilde{L}_t^\infty(\mathbf{\mathcal{B}}_{2,1}^{3/2})}}\Vert \mathbf{U}(t)\Vert_{\widetilde{L}^2_t(\mathbb{B}_{2,1}^{3/2})}. 
\end{aligned}
\end{equation}
\end{proposition}
The proof of Proposition \ref{Main_Proposition} will be given through several lemmas and uses some ideas and estimates from the work \cite{Racke_Said_2019}. 
First, we recall  the following estimate which has been proved in \cite[Estimate (2.39)]{Racke_Said_2019}. 
\begin{lemma}[\cite{Racke_Said_2019}] Assume that $0<\tau<\beta$. Then, the following estimate holds: \begin{eqnarray}  \label{Main_Estimate_D_0_1}
\mathcal{E}^2[\mathbf{U}](t)+\mathcal{D}^2[\mathbf{U}](t)\leq \mathcal{E}^2[\mathbf{U}](0) + C\big(\mathcal{E}%
[\mathbf{U}](t)+M_0(t)\big)\mathcal{D}^2[\mathbf{U}](t),\end{eqnarray} 
uniformly with respect to $t$. 
\end{lemma}

From \eqref{Main_Estimate_D_0_1} and recalling \eqref{Weighted_Energy} and \eqref{Dissipative_weighted_norm_1}, we get 
\begin{equation}\label{Main_Estimate_Order_0}
\begin{aligned}
&\Vert \mathbf{U}(t)\Vert_{L_t^\infty(\mathcal{L}^2)}+\Vert \mathbf{U}(t)\Vert_{L^2_t(\mathbb{L}^2)}\\
\lesssim&\, \Vert \mathbf{U}(0)\Vert_{\mathcal{L}^2}
+\sqrt{\Vert \mathbf{U}(t)\Vert_{L_t^\infty(\mathcal{L}^2)}+M_0(t)}\Vert \mathbf{U}(t)\Vert_{L^2_t(\mathbb{L}^2)}. 
\end{aligned}
\end{equation}

Our goal now is to prove a frequency-localized estimate similar to \eqref{Main_Estimate_Order_0}. This is the goal of Section \ref{Sec:Frequency_Local_Est}.

\subsection{Frequency-localized estimates }
\label{Sec:Frequency_Local_Est}
In this subsection,  we employ a frequency-localization method  and prove a nonlinear energy estimate in homogeneous Besov-type spaces. Applying the operator $\dot{\Delta}_q\, (q\in \Z)$ to the system \eqref{System_New}, we get 
\begin{equation}
\left\{
\begin{array}{ll}
\dot{\Delta}_qu_{t}=\dot{\Delta}_qv,\vspace{0.2cm} &  \\
\dot{\Delta}_qv_{t}=\dot{\Delta}_qw,\vspace{0.2cm} &  \\
\tau \dot{\Delta}_qw_{t}=\Delta \dot{\Delta}_qu+\beta \Delta \dot{\Delta}_qv-\dot{\Delta}_qw+
\mathrm{R}_q  &
\end{array}%
\right.  \label{System_New_Localized}
\end{equation}
with 
\begin{equation}  \label{R_q}
\begin{aligned}
\mathrm{R}_q=&\,\dot{\Delta}_q \Big(\dfrac{B}{A}vw+2\nabla u\nabla v\Big)\\
=&\,\dfrac{B}{A}[\dot{\Delta}_q,v]w+\dfrac{B}{A}v\dot{\Delta}_qw+2[\dot{\Delta}_q,\nabla u]\nabla v+2\nabla u\nabla \dot{\Delta}_qv.
\end{aligned}
\end{equation}
We define the frequency-localized energy  as 
\begin{equation}\label{Weighted_Energy_Localized}
\begin{aligned}
\mathcal{E}^2[\dot{\Delta}_q \mathbf{U}](t)=\,\sup_{0\leq \sigma\leq t}&\Big(\big\Vert \dot{\Delta}_q(v+\tau
w)(\sigma)\big\Vert _{L^{2}}^{2}+\big\Vert \dot{\Delta}_q\nabla (v+\tau
w)(\sigma)\big\Vert _{L^{2}}^{2}\Big.%
\vspace{0.2cm}  \notag \\
 \Big.&+\big\Vert \dot{\Delta}_q \Delta  v(\sigma)\big\Vert
_{L^{2}}^{2}+\big\Vert\dot{\Delta}_q \nabla  v(\sigma)\big\Vert _{L^{2}}^{2}+\big\Vert \dot{\Delta}_q \Delta (u+\tau v)(\sigma)\big\Vert
_{L^{2}}^{2}\Big.\\
\Big.&+\big\Vert \dot{\Delta}_q\nabla (u+\tau v)(\sigma)\big\Vert
_{L^{2}}^{2}+\Vert \dot{\Delta}_q w(\sigma)\Vert_{L^2}^2\Big),
\end{aligned}
\end{equation}
and its associated dissipative rate as 
\begin{equation}\label{Dissipation_Localized}
\begin{aligned}
\mathscr{D}^2[\dot{\Delta}_q \mathbf{U}](t)=&\,\Big(\big\Vert \dot{\Delta}_q  \nabla
v(t)\big\Vert _{L^{2}}^{2}+\big\Vert \dot{\Delta}_q \Delta 
v(t)\big\Vert
_{L^{2}}^{2}+ \Vert \dot{\Delta}_q  w (t)\Vert_{L^2}^2\Big.   \\
\Big.&+\big\Vert \dot{\Delta}_q \Delta \left( u+\tau v\right)(t) \big\Vert
_{L^{2}}^{2} + \big\Vert \dot{\Delta}_q \nabla (v+\tau w)(t)\big\Vert _{L^{2}}^{2}%
\Big).
\end{aligned}
\end{equation}
We also define 

\begin{equation}\label{Dissipative_weighted_norm_2}
\mathcal{D}^{2}[\dot{\Delta}_q \mathbf{U}](t) =\int_0^t\mathscr{D}^2[\dot{\Delta}_q \mathbf{U}](t)(\sigma)\textup{d}\sigma .
\end{equation}
Following the same steps as in \cite{Racke_Said_2019}, where $\dot{\Delta}_q$ will play the role of $\nabla^k$ in \cite{Racke_Said_2019}, we obtain 
\begin{equation}\label{Main_Freq_Energy}
\begin{aligned}
\mathcal{E}^2[\dot{\Delta}_q \mathbf{U}](t)+\mathcal{D}^{2}[\dot{\Delta}_q \mathbf{U}](t)\lesssim \mathcal{E}^2[\dot{\Delta}_q \mathbf{U}](0)+\sum_{i=1}^5\int_0^t \mathrm{\mathbf{I}}_i[\dot{\Delta}_q \mathbf{U}](\sigma)\textup{d}\sigma,
\end{aligned}  
\end{equation}
with 
\begin{equation}\label{Renmaining_Terms}
\begin{aligned}
\mathrm{\mathbf{I}}_1[\dot{\Delta}_q \mathbf{U}]=&\,\int_{\mathbb{R}^{3}}|\mathrm{R}_q||( \dot{\Delta}_qv+\tau \dot{\Delta}_qw)|
\textup{d}x,\\
\mathrm{\mathbf{I}}_2[\dot{\Delta}_q \mathbf{U}]=&\,\int_{\mathbb{R}^{3}}|\nabla \mathrm{R}_q|| \nabla (\dot{\Delta}_qv+\tau \dot{\Delta}_qw)|\textup{d}x,\\
\mathrm{\mathbf{I}}_3[\dot{\Delta}_q \mathbf{U}]=&\,\int_{\mathbb{R}^{3}}|\mathrm{R}_q||\Delta ( \dot{\Delta}_q u+\tau
\dot{\Delta}_q v)| \textup{d}x,\\
\mathrm{\mathbf{I}}_4[\dot{\Delta}_q \mathbf{U}]=&\,\int_{\mathbb{R}^{3}}|\nabla \mathrm{R}_q| |\nabla \dot{\Delta}_q v|\textup{d}x,\\
\mathrm{\mathbf{I}}_5[\dot{\Delta}_q \mathbf{U}]=&\,\int_{\mathbb{R}^{3}}|\mathrm{R}_q||\dot{\Delta}_qw|\textup{d}x. 
\end{aligned}
\end{equation}
Our goal now is to estimate the terms $\mathrm{\mathbf{I}}_i[\dot{\Delta}_q \mathbf{U}],\, i=1,\dots,5$. 
This will be done by a repeated use of suitable  functional inequalities, such as:   Gagliardo--Nirenberg inequalities and some  embedding theorems in Besov spaces.  These functional inequalities  will help to control some commutator estimates. 
The proof will be done through several lemmas. 

 In the following lemma, we estimate $\mathrm{\mathbf{I}}_2[\dot{\Delta}_q \mathbf{U}] $ and $\mathrm{\mathbf{I}}_4[\dot{\Delta}_q \mathbf{U}]$. 
 \begin{lemma}\label{Lemma_I_2}
 It holds that 
 \begin{equation}\label{I_2_4_Main_Estimate}
\begin{aligned}
&\int_0^t\Big(\mathrm{\mathbf{I}}_2[\dot{\Delta}_q \mathbf{U}]  +\mathrm{\mathbf{I}}_4[\dot{\Delta}_q \mathbf{U}]\Big)(\sigma)\textup{d}\sigma\\
\lesssim&\, c_q 2^{-\frac{3q}{2}}\Big(\Vert v\Vert_{\widetilde{L}_t^{2}(\dot{B}^{3/2}_{2,1})}\Vert \nabla w\Vert_{\widetilde{L}_t^{\infty}(\dot{B}^{1/2}_{2,1})}+\Vert \nabla v\Vert_{\widetilde{L}_t^{2}(\dot{B}^{3/2}_{2,1})}\Vert  w\Vert_{\widetilde{L}_t^{\infty}(\dot{B}^{1/2}_{2,1})}\Big.\\
\Big.&+\Vert \nabla^2 u\Vert_{\widetilde{L}_t^{2}(\dot{B}^{3/2}_{2,1})}\Vert\nabla v\Vert_{\widetilde{L}_t^{\infty}(\dot{B}^{1/2}_{2,1})}+\Vert \nabla u\Vert_{\widetilde{L}_t^{2}(\dot{B}^{3/2}_{2,1})}\Vert \nabla^2 v\Vert_{\widetilde{L}_t^{\infty}(\dot{B}^{1/2}_{2,1})}\Big)\mathcal{D}[\dot{\Delta}_q \mathbf{U}](t)\\
&+ \Big(\Vert v\Vert_{L_t^\infty(\dot{B}^{3/2}_{2,1})}+\Vert \nabla v\Vert_{L_t^\infty(\dot{B}^{3/2}_{2,1})}+\Vert \nabla u\Vert_{L_t^\infty(\dot{B}^{3/2}_{2,1})}+\Vert \nabla^2 u\Vert_{\widetilde{L}_t^\infty(\dot{B}^{3/2}_{2,1})}\Big)\mathcal{D}^{2}[\dot{\Delta}_q \mathbf{U}](t).
\end{aligned}
\end{equation}

 \end{lemma}
 \begin{proof}
 We have by H\"older's inequality 
 \begin{equation}\label{I_2_4_Main}
\begin{aligned}
&\int_0^t\Big(\mathrm{\mathbf{I}}_2[\dot{\Delta}_q \mathbf{U}](\sigma) +\mathrm{\mathbf{I}}_4[\dot{\Delta}_q \mathbf{U}](\sigma)\Big)\textup{d}\sigma \\\lesssim&\, \Vert \nabla \mathrm{R}_q \Vert_{L^2_t(L^2)}\Big(\Vert \nabla (\dot{\Delta}_qv+\tau \dot{\Delta}_qw)\Vert_{L^2_t(L^2)}+\Vert\nabla \dot{\Delta}_q v \Vert_{L^2_t(L^2)}\Big). 
\end{aligned}
\end{equation}
To estimate $\Vert \nabla \mathrm{R}_q \Vert_{L^2}$,
we write  
\begin{equation}\label{nabla_R_q}
\begin{aligned}
\nabla\mathrm{R}_q=&\,\dot{\Delta}_q \nabla\Big(\dfrac{B}{A}vw+2\nabla u\nabla v\Big)\\
=&\,\dot{\Delta}_q \Big(\dfrac{B}{A} v\nabla w+\dfrac{B}{A} \nabla v w+2\nabla^2 u\nabla v+2\nabla u\nabla^2v\Big)\\
=&\,\dfrac{B}{A}[\dot{\Delta}_q,v]\nabla w+\dfrac{B}{A}v\dot{\Delta}_q\nabla w+\dfrac{B}{A}[\dot{\Delta}_q,\nabla v] w+\dfrac{B}{A}\nabla v\dot{\Delta}_q w\\
&+2[\dot{\Delta}_q,\nabla^2 u]\nabla v+2\nabla^2 u\dot{\Delta}_q \nabla v+2[\dot{\Delta}_q,\nabla u]\nabla^2 v+2\nabla u\dot{\Delta}_q\nabla^2 v. 
\end{aligned}
\end{equation}
 Applying \eqref{Commutator_2}, we estimate the commutators in \eqref{nabla_R_q}  as 
\begin{equation}\label{Comm_Estimates}
\begin{aligned}
\Vert[\dot{\Delta}_q,v]\nabla w\Vert_{L^2_t(L^2)}\lesssim &\, c_q 2^{-\frac{3q}{2}}\Vert v\Vert_{\widetilde{L}_t^{2}(\dot{B}^{3/2}_{2,1})}\Vert \nabla w\Vert_{\widetilde{L}_t^{\infty}(\dot{B}^{1/2}_{2,1})}\\
\Vert[\dot{\Delta}_q,\nabla v] w\Vert_{L^2_t(L^2)}\lesssim&\, c_q 2^{-\frac{3q}{2}}\Vert \nabla v\Vert_{\widetilde{L}_t^{2}(\dot{B}^{3/2}_{2,1})}\Vert  w\Vert_{\widetilde{L}_t^{\infty}(\dot{B}^{1/2}_{2,1})}\\
\Vert[\dot{\Delta}_q,\nabla^2 u] \nabla v\Vert_{L^2_t(L^2)}\lesssim&\, c_q 2^{-\frac{3q}{2}}\Vert \nabla^2 u\Vert_{\widetilde{L}_t^{2}(\dot{B}^{3/2}_{2,1})}\Vert \nabla v\Vert_{\widetilde{L}_t^{\infty}(\dot{B}^{1/2}_{2,1})}\\
\Vert[\dot{\Delta}_q,\nabla u] \nabla^2 v\Vert_{L^2_t(L^2)}\lesssim&\, c_q 2^{-\frac{3q}{2}}\Vert \nabla u\Vert_{\widetilde{L}_t^{2}(\dot{B}^{3/2}_{2,1})}\Vert \nabla^2 v\Vert_{\widetilde{L}_t^{\infty}(\dot{B}^{1/2}_{2,1})}.
\end{aligned}
\end{equation}
We point out here and in the sequel that each $(c_q)_{q\in \Z}$ has possibly a different form, however we always have $\Vert c_q\Vert_{\ell^1}\leq 1.$

\noindent On the other hand, we estimate the remaining terms in \eqref{nabla_R_q} as 
  \begin{equation}\label{Prod_Estimates}
\begin{aligned}
\Vert v\dot{\Delta}_q\nabla w\Vert_{L^2_t(L^2)}\lesssim &\,\Vert v\Vert_{L_t^\infty(L^\infty)}\Vert \dot{\Delta}_q\nabla w\Vert_{L_t^2(L^2)}\\
\Vert \nabla v\dot{\Delta}_q w\Vert_{L^2_t(L^2)}\lesssim &\,\Vert \nabla v\Vert_{L_t^\infty(L^\infty)}\Vert \dot{\Delta}_q w\Vert_{L_t^2(L^2)}\\
\Vert \nabla^2 u\dot{\Delta}_q \nabla v\Vert_{L^2_t(L^2)}\lesssim &\,\Vert \nabla^2 u\Vert_{L_t^\infty(L^\infty)}\Vert \dot{\Delta}_q\nabla v\Vert_{L_t^2(L^2)}\\
\Vert \nabla u\dot{\Delta}_q\nabla^2 v\Vert_{L^2_t(L^2)}\lesssim &\,\Vert \nabla u\Vert_{L_t^\infty(L^\infty)}\Vert \dot{\Delta}_q\nabla^2 v\Vert_{L_t^2(L^2)}. 
\end{aligned}
\end{equation}
Plugging \eqref{Comm_Estimates} and \eqref{Prod_Estimates} into \eqref{I_2_4_Main} and  using  the embedding 
$\dot{B}_{2,1}^{3/2}(\R^3)\hookrightarrow L^\infty (\R^3)$, then \eqref{I_2_4_Main_Estimate} holds. This finishes the proof of Lemma \ref{Lemma_I_2}. 
 \end{proof}
 
 Next, we estimate $\mathrm{\mathbf{I}}_i[\dot{\Delta}_q \mathbf{U}],\, i=3,5$. We have the following lemma. 
 
 \begin{lemma}\label{Lemma_I_3}
 It holds that 
 \begin{equation}\label{I_3_5_Main_Estimate}
\begin{aligned}
&\int_0^t\Big(\mathrm{\mathbf{I}}_3[\dot{\Delta}_q \mathbf{U}]  +\mathrm{\mathbf{I}}_5[\dot{\Delta}_q \mathbf{U}]\Big)(\sigma)\textup{d}\sigma\\
\lesssim&\, c_q 2^{-\frac{3q}{2}}\Big(\Vert v\Vert_{\widetilde{L}_t^{2}(\dot{B}^{3/2}_{2,1})}\Vert w\Vert_{\widetilde{L}_t^{\infty}(\dot{B}^{1/2}_{2,1})}
+\Vert \nabla u\Vert_{\widetilde{L}_t^{2}(\dot{B}^{3/2}_{2,1})}\Vert\nabla v\Vert_{\widetilde{L}_t^{\infty}(\dot{B}^{1/2}_{2,1})}\Big)\mathcal{D}[\dot{\Delta}_q \mathbf{U}](t)\\
&+ \Big(\Vert v\Vert_{L_t^\infty(\dot{B}^{3/2}_{2,1})}+\Vert \nabla u\Vert_{L_t^\infty(\dot{B}^{3/2}_{2,1})}\Big)\mathcal{D}^{2}[\dot{\Delta}_q \mathbf{U}](t).  
\end{aligned}
\end{equation}
 \end{lemma}
 
 \begin{proof}
 We have, by recalling \eqref{Renmaining_Terms} and using H\"older's inequality 
 \begin{equation}\label{I_3_5_Main}
\begin{aligned}
&\int_0^t\Big(\mathrm{\mathbf{I}}_3[\dot{\Delta}_q \mathbf{U}](\sigma) +\mathrm{\mathbf{I}}_5[\dot{\Delta}_q \mathbf{U}](\sigma)\Big)\textup{d}\sigma \\\lesssim&\, \Vert \mathrm{R}_q \Vert_{L^2_t(L^2)}\Big(\Vert \Delta ( \dot{\Delta}_q u+\tau
\dot{\Delta}_q v)\Vert_{L^2_t(L^2)}+\Vert\dot{\Delta}_q w \Vert_{L^2_t(L^2)}\Big). 
\end{aligned}
\end{equation}
Hence,  following  the same steps as in the proof of Lemma \ref{Lemma_I_2}, we obtain 
 \begin{equation}\label{R_1_Comm_3}
\begin{aligned}
\Vert  \mathrm{R}_q \Vert_{L^2_t(L^2)}\lesssim &\,c_q 2^{-\frac{3q}{2}}\Big(\Vert v\Vert_{\widetilde{L}_t^{2}(\dot{B}^{3/2}_{2,1})}\Vert w\Vert_{\widetilde{L}_t^{\infty}(\dot{B}^{1/2}_{2,1})}
+\Vert \nabla u\Vert_{\widetilde{L}_t^{2}(\dot{B}^{3/2}_{2,1})}\Vert\nabla v\Vert_{\widetilde{L}_t^{\infty}(\dot{B}^{1/2}_{2,1})}\Big)\\
&+\Big(\Vert v\Vert_{L_t^\infty(L^\infty)}\Vert \dot{\Delta}_qw\Vert_{L_t^2(L^2)}+\Vert \nabla u\Vert_{L_t^\infty(L^\infty)}\Vert \dot{\Delta}_q\nabla v\Vert_{L_t^2(L^2)}\Big). 
\end{aligned}
\end{equation}
Therefore, \eqref{I_3_5_Main_Estimate} holds by collecting the above last two estimates and using  the embedding 
$\dot{B}_{2,1}^{3/2}(\R^3)\hookrightarrow L^\infty (\R^3)$, which ends the proof of Lemma \ref{Lemma_I_3}.   
\end{proof}

Now, we need to estimate $\mathrm{\mathbf{I}}_1[\dot{\Delta}_q \mathbf{U}]$. This will be done in the next lemma. 
\begin{lemma}\label{Lemma_I_1}
It holds that 
\begin{equation}\label{I_1_Estimate}
\begin{aligned}
&\int_0^t \mathrm{\mathbf{I}}_1[\dot{\Delta}_q \mathbf{U}](\sigma)\textup{d}\sigma\\
\lesssim&\,\big(\mathcal{E}[\mathbf{U}](t)+\Vert \nabla u\Vert_{\widetilde{L}_t^\infty(\dot{B}^{1/2}_{2,1})}+\Vert v \Vert_{\widetilde{L}_t^\infty(\dot{B}^{1/2}_{2,1}}\big)\mathcal{D}^2[\dot{\Delta}_q \mathbf{U}](t)\\
&+ c_q 2^{-\frac{3q}{2}}\mathcal{E}[\dot{\Delta}_q \mathbf{U}](t)\Big(\Vert v\Vert_{\widetilde{L}_t^2(\dot{B}^{3/2}_{2,1})}\Vert w\Vert_{\widetilde{L}_t^2(\dot{B}^{1/2}_{2,1})}
+\Vert \nabla u\Vert_{\widetilde{L}_t^2(\dot{B}^{3/2}_{2,1})}\Vert \nabla  v\Vert_{\widetilde{L}_t^2(\dot{B}^{1/2}_{2,1})}\Big). 
\end{aligned}
\end{equation}
\end{lemma}

\begin{proof}
We have by recalling \eqref{R_q}, 
\begin{equation}
\begin{aligned}
\mathrm{\mathbf{I}}_1[\dot{\Delta}_q \mathbf{U}]\lesssim &\,\int_{\mathbb{R}^{3}}|[\dot{\Delta}_q,v]w||( \dot{\Delta}_qv+\tau \dot{\Delta}_qw)|
\textup{d}x +\int_{\mathbb{R}^{3}}|v\dot{\Delta}_qw||( \dot{\Delta}_qv+\tau \dot{\Delta}_qw)|
\textup{d}x \\
&+\int_{\mathbb{R}^{3}}|[\dot{\Delta}_q,\nabla u]\nabla v||( \dot{\Delta}_qv+\tau \dot{\Delta}_qw)|
\textup{d}x +\int_{\mathbb{R}^{3}}|\nabla u\nabla \dot{\Delta}_qv||( \dot{\Delta}_qv+\tau \dot{\Delta}_qw)|
\textup{d}x\\
:=&\,\mathrm{\mathbf{T}}_1+\mathrm{\mathbf{T}}_2+\mathrm{\mathbf{T}}_3+\mathrm{\mathbf{T}}_4. 
\end{aligned}
\end{equation}
Our next goal is to estimate the terms $\mathrm{\mathbf{T}}_i,\, i=1,\dots,4$. We begin  
by  estimating  $\mathrm{\mathbf{T}}_4$. We have 
\begin{equation}
\begin{aligned}
\mathrm{\mathbf{T}}_4\lesssim &\,\int_{\mathbb{R}^{3}}|\nabla u\nabla \dot{\Delta}_qv|| \dot{\Delta}_qv| 
\textup{d}x+\int_{\mathbb{R}^{3}}|\nabla u\nabla \dot{\Delta}_qv|| \dot{\Delta}_qw| 
\textup{d}x\\
:=&\,\mathrm{\mathbf{J}}_1+\mathrm{\mathbf{J}}_2.
\end{aligned}
\end{equation}
 To estimate $\mathrm{\mathbf{J}}_1$, we have 
  by H\"older's inequality
\begin{eqnarray}  \label{J_1_Estimate}
\mathrm{\mathbf{J}}_1\lesssim  \Vert \dot{\Delta}_q v\Vert_{L^6}\Vert \nabla u\Vert_{L^3}\Vert \nabla
\dot{\Delta}_qv\Vert_{L^2}.
\end{eqnarray}
Now, applying the interpolation inequality, which holds for $n=3$, 
\begin{eqnarray}\label{Estimate_f_N_3}
\Vert f\Vert_{L^3}\leq C \Vert f\Vert_{L^2}^{1/2}\Vert \nabla
f\Vert_{L^2}^{1/2},
\end{eqnarray} 
we obtain
\begin{eqnarray}\label{Nabla_1_u}
\Vert \nabla u\Vert_{L^3}\lesssim \Vert \nabla u\Vert_{L^2}^{1/2}\Vert
\nabla^2 u\Vert_{L^2}^{1/2}.
\end{eqnarray}
This together with the Sobolev embedding $\dot{H}^1(\R^3)\hookrightarrow L^6(\R^3)$,
 yields 
  \begin{equation}  \label{J_1_Estimate_Main}
  \begin{aligned}
\mathrm{\mathbf{J}}_1\lesssim& \,\Vert \nabla u\Vert_{L^2}^{1/2}\Vert \nabla^2 u\Vert_{L^2}^{1/2}
\Vert \nabla \dot{\Delta}_q  v\Vert_{L^2}^2  \\
\lesssim&\, \big(\Vert \nabla u\Vert_{L^2}+\Vert \nabla^2 u\Vert_{L^2}\big)\Vert \nabla\dot{\Delta}_q 
v\Vert_{L^2}^2.
\end{aligned}
\end{equation}
The estimate of $\mathrm{\mathbf{J}}_2$ is  straightforward, so, we have 
\begin{equation}\label{J_2_Estimate_Main}
\begin{aligned}
\mathrm{\mathbf{J}}_2\lesssim&\, \Vert \nabla u\Vert_{L^\infty}\Vert \nabla \dot{\Delta}_q v\Vert_{L^2}\Vert
\dot{\Delta}_qw\Vert_{L^2}\\
\lesssim&\,\Vert \nabla u\Vert_{\dot{B}^{3/2}_{2,1}}\Vert \nabla \dot{\Delta}_q v\Vert_{L^2}\Vert
\dot{\Delta}_qw\Vert_{L^2}. 
\end{aligned}
\end{equation}
Consequently, collecting \eqref{J_1_Estimate_Main} and \eqref{J_2_Estimate_Main} and using H\"older's inequality,   we obtain 
\begin{equation}
\begin{aligned}
\mathrm{\mathbf{T}}_4\lesssim &\,\big(\Vert \nabla u\Vert_{L^2}+\Vert \nabla^2 u\Vert_{L^2}+\Vert \nabla u\Vert_{\dot{B}^{3/2}_{2,1}}\big)\Big(\Vert \nabla\dot{\Delta}_q 
v\Vert_{L^2}^2+\Vert
\dot{\Delta}_qw\Vert_{L^2}^2\Big).
\end{aligned}
\end{equation}
Consequently,  this yields, by using \eqref{Mink},  
\begin{equation}\label{T_4_Estimate}
\begin{aligned}
\int_0^t\mathrm{\mathbf{T}}_4(\sigma)\textup{d}\sigma\lesssim&\,\big(\mathcal{E}[\mathbf{U}](t)+\Vert \nabla u\Vert_{\widetilde{L}_t^\infty(\dot{B}^{3/2}_{2,1})}\big)\mathcal{D}^2[\dot{\Delta}_q \mathbf{U}](t). 
\end{aligned}
\end{equation}
Next, we estimate $\mathrm{\mathbf{T}}_2$. We have 
\begin{equation}
\begin{aligned}
\mathrm{\mathbf{T}}_2=&\,\int_{\mathbb{R}^{3}}|v\dot{\Delta}_qw||( \dot{\Delta}_qv+\tau \dot{\Delta}_qw)|\textup{d}x\\
\lesssim &\,\int_{\mathbb{R}^{3}} |v\dot{\Delta}_qw| |\dot{\Delta}_qv|\textup{d}x+\int_{\mathbb{R}^{3}} |v||\dot{\Delta}_qw|^2 \textup{d}x\\
:=&\, \mathrm{\mathbf{J}}_3+\mathrm{\mathbf{J}}_4. 
\end{aligned}
\end{equation}
Using \eqref{Estimate_f_N_3}, 
we obtain 
\begin{equation}\label{J_3_Main}
\begin{aligned}
\mathrm{\mathbf{J}}_3\lesssim &\,\Vert v\Vert_{L^3}\Vert   \dot{\Delta}_qv\Vert_{L^6}\Vert\dot{\Delta}_qw\Vert_{L^2}\\
\lesssim&\, \Vert v\Vert_{L^2}^{1/2}\Vert \nabla v\Vert_{L^2}^{1/2}\Vert \dot{\Delta}_q\nabla v\Vert_{L^2}\Vert\dot{\Delta}_qw\Vert_{L^2}\\
\lesssim&\,(\Vert v\Vert_{L^2}+\Vert \nabla v\Vert_{L^2})\Big(\Vert \dot{\Delta}_q\nabla v\Vert_{L^2}^2+\Vert\dot{\Delta}_qw\Vert_{L^2}^2\Big). 
\end{aligned}
\end{equation}
Furthermore,  we estimate $\mathrm{\mathbf{J}}_4$ as 
\begin{equation}\label{J_4_Main}
\begin{aligned}
\mathrm{\mathbf{J}}_4\lesssim &\,\Vert v\Vert_{L^\infty}\Vert\dot{\Delta}_qw\Vert_{L^2}^2
\lesssim \Vert v \Vert_{\dot{B}^{3/2}_{2,1}}\Vert\dot{\Delta}_qw\Vert_{L^2}^2. 
\end{aligned}
\end{equation}
Collecting \eqref{J_3_Main} and \eqref{J_4_Main}, we obtain 
\begin{equation}
\mathrm{\mathbf{T}}_2\lesssim \big(\Vert v\Vert_{L^2}+\Vert \nabla v\Vert_{L^2}+\Vert v \Vert_{\dot{B}^{3/2}_{2,1}}\big)\Big(\Vert \dot{\Delta}_q\nabla v\Vert_{L^2}^2+\Vert\dot{\Delta}_qw\Vert_{L^2}^2\Big). 
\end{equation}
This yields, by using H\"older's inequality and \eqref{Mink}, 
\begin{equation}\label{T_2_Main}
\begin{aligned}
\int_0^t\mathrm{\mathbf{T}}_2(\sigma)\textup{d}\sigma\lesssim \big(\mathcal{E}[\mathbf{U}](t)+\Vert v \Vert_{\widetilde{L}_t^\infty(\dot{B}^{3/2}_{2,1}}\big) \mathcal{D}^2[\dot{\Delta}_q \mathbf{U}](t).
\end{aligned}
\end{equation}
Now, to estimate $\mathrm{\mathbf{T}}_1$,  we have 
\begin{equation}\label{T_1_Com}
\begin{aligned}  
\mathrm{\mathbf{T}}_1=&\,\int_{\mathbb{R}^{3}}|[\dot{\Delta}_q,v]w||( \dot{\Delta}_qv+\tau \dot{\Delta}_qw)|\textup{d}x\\
\lesssim&\,
\big\Vert [\dot{\Delta}_q,v]w\big\Vert _{L^{2}}\big\Vert (
\dot{\Delta}_qv+\tau \dot{\Delta}_qw) \big\Vert _{L^{2}}.
\end{aligned}
\end{equation}
Applying \eqref{Commu_A_1}, we find  
\begin{equation}
\begin{aligned}
\big\Vert [\dot{\Delta}_q,v]w\big\Vert _{L^{2}}\lesssim  c_q 2^{-\frac{3q}{2}}\Vert v\Vert_{\dot{B}^{3/2}_{2,1}}\Vert w\Vert_{\dot{B}^{1/2}_{2,1}}.
\end{aligned}
\end{equation}
This yields 
\begin{equation}\label{T_1_Main}
\begin{aligned}
\int_0^t\mathrm{\mathbf{T}}_1(\sigma) \textup{d}\sigma\lesssim&\, c_q 2^{-\frac{3q}{2}}\big\Vert (
\dot{\Delta}_qv+\tau \dot{\Delta}_qw)\big\Vert _{L_t^\infty(L^{2})}\Vert v\Vert_{L_t^2(\dot{B}^{3/2}_{2,1})}\Vert w\Vert_{L_t^2(\dot{B}^{1/2}_{2,1})}\\
\lesssim&\,c_q 2^{-\frac{3q}{2}}\big\Vert (
\dot{\Delta}_qv+\tau \dot{\Delta}_qw)\big\Vert _{L_t^\infty(L^{2})}\Vert v\Vert_{\widetilde{L}_t^2(\dot{B}^{3/2}_{2,1})}\Vert w\Vert_{\widetilde{L}_t^2(\dot{B}^{1/2}_{2,1})}\\
\lesssim&\,c_q 2^{-\frac{3q}{2}}\mathcal{E}[\dot{\Delta}_q \mathbf{U}](t)\Vert v\Vert_{\widetilde{L}_t^2(\dot{B}^{3/2}_{2,1})}\Vert w\Vert_{\widetilde{L}_t^2(\dot{B}^{1/2}_{2,1})},
\end{aligned}
\end{equation}
where we have used \eqref{Mink}.

\noindent The term $\mathrm{\mathbf{T}}_3$  can be estimated similarly. We have 
\begin{equation}
\begin{aligned}
\mathrm{\mathbf{T}}_3 \lesssim &\,\big\Vert [\dot{\Delta}_q,\nabla u]\nabla v\big\Vert _{L^{2}}\big\Vert (
\dot{\Delta}_qv+\tau \dot{\Delta}_qw) \big\Vert _{L^{2}}\\
\lesssim&\, c_q 2^{-\frac{3q}{2}}\Vert \nabla u\Vert_{\dot{B}^{3/2}_{2,1}}\Vert \nabla v\Vert_{\dot{B}^{1/2}_{2,1}}\big\Vert (
\dot{\Delta}_qv+\tau \dot{\Delta}_qw) \big\Vert _{L^{2}}.   
\end{aligned}
\end{equation}
This yields as above, 
\begin{equation}\label{T_3_Main}
\begin{aligned}
\int_0^t\mathrm{\mathbf{T}}_3(\sigma) \textup{d}\sigma\lesssim&\,c_q 2^{-\frac{3q}{2}}\big\Vert (
\dot{\Delta}_qv+\tau \dot{\Delta}_qw)\big\Vert _{L_t^\infty(L^{2})}\Vert \nabla u\Vert_{\widetilde{L}_t^2(\dot{B}^{3/2}_{2,1})}\Vert \nabla  v\Vert_{\widetilde{L}_t^2(\dot{B}^{1/2}_{2,1})}\\
\lesssim&\,c_q 2^{-\frac{3q}{2}}\mathcal{E}[\dot{\Delta}_q \mathbf{U}](t)\Vert \nabla u\Vert_{\widetilde{L}_t^2(\dot{B}^{3/2}_{2,1})}\Vert \nabla  v\Vert_{\widetilde{L}_t^2(\dot{B}^{1/2}_{2,1})}. 
\end{aligned}
\end{equation}
Collecting the estimates \eqref{T_4_Estimate}, \eqref{T_2_Main}, \eqref{T_1_Main} and \eqref{T_3_Main} and using H\"older's inequality with respect to $t$, we obtain \eqref{I_1_Estimate}. This finishes the proof of Lemma \ref{Lemma_I_1}.   
 \end{proof}
 
 \begin{proof}[Proof of Proposition \ref{Main_Proposition}]
 Using \eqref{Eqv_Norms} together with \eqref{Mink} and recalling  \eqref{M_Functional},  we have 
 \begin{equation}\label{L_infty_M_Estimate}
\Vert v\Vert_{L_t^\infty(\dot{B}^{3/2}_{2,1})}+\Vert \nabla v\Vert_{L_t^\infty(\dot{B}^{3/2}_{2,1})}+\Vert \nabla u\Vert_{L_t^\infty(\dot{B}^{3/2}_{2,1})}+\Vert \nabla^2 u\Vert_{\widetilde{L}_t^\infty(\dot{B}^{3/2}_{2,1})}\lesssim M(t). 
\end{equation}
 Now, plugging \eqref{I_2_4_Main_Estimate}, \eqref{I_3_5_Main_Estimate} and  \eqref{I_1_Estimate} into \eqref{Main_Freq_Energy} and exploiting \eqref{L_infty_M_Estimate}, 
 we obtain 
 \begin{equation}\label{Main_Freq_Energy_2}
\begin{aligned}
&\quad\mathcal{E}^2[\dot{\Delta}_q \mathbf{U}](t)+\mathcal{D}^{2}[\dot{\Delta}_q \mathbf{U}](t)\\
&\lesssim \mathcal{E}^2[\dot{\Delta}_q \mathbf{U}](0)
+\big(\mathcal{E}[\mathbf{U}](t)+M(t)\big)\mathcal{D}^2[\dot{\Delta}_q \mathbf{U}](t)\\
&\,+  c_q 2^{-\frac{3q}{2}}\mathcal{E}[\dot{\Delta}_q \mathbf{U}](t)\Big(\Vert v\Vert_{\widetilde{L}_t^2(\dot{B}^{3/2}_{2,1})}\Vert w\Vert_{\widetilde{L}_t^2(\dot{B}^{1/2}_{2,1})}
+\Vert \nabla u\Vert_{\widetilde{L}_t^2(\dot{B}^{3/2}_{2,1})}\Vert \nabla  v\Vert_{\widetilde{L}_t^2(\dot{B}^{1/2}_{2,1})}\Big)\\
&\,+c_q 2^{-\frac{3q}{2}}\Big(\Vert v\Vert_{\widetilde{L}_t^{2}(\dot{B}^{3/2}_{2,1})}+\Vert \nabla v\Vert_{\widetilde{L}_t^{2}(\dot{B}^{3/2}_{2,1})}
+\Vert \nabla u\Vert_{\widetilde{L}_t^{2}(\dot{B}^{3/2}_{2,1})}
+\Vert \nabla^2 u\Vert_{\widetilde{L}_t^{2}(\dot{B}^{3/2}_{2,1})}\Big)M(t)\mathcal{D}[\dot{\Delta}_q \mathbf{U}](t). 
\end{aligned}  
\end{equation}
Using \eqref{Eqv_Norms}, we obtain  
\begin{equation}
\begin{aligned}
\Vert v\Vert_{\widetilde{L}_t^2(\dot{B}^{3/2}_{2,1})}+\Vert \nabla u\Vert_{\widetilde{L}_t^2(\dot{B}^{3/2}_{2,1})}\lesssim &\,\Vert \nabla v\Vert_{\widetilde{L}_t^2(\dot{B}^{1/2}_{2,1})}+\Vert \nabla^2 u\Vert_{\widetilde{L}_t^2(\dot{B}^{1/2}_{2,1})}\\
\lesssim&\,\Vert \mathbf{U}\Vert_{\widetilde{L}^2_t(\dot{\mathbb{B}}_{2,1}^{1/2})}. 
\end{aligned}
\end{equation}
Multiplying the above identity by $2^{3q}$, using Young's inequality  and  the elementary  inequality $\sqrt{a+b}\leq \sqrt{a}+\sqrt{b}\leq \sqrt{2(a+b)}, \, a,b>0$, we obtain  
\begin{equation}\label{Main_Freq_Energy_3}
\begin{aligned}
&2^{\frac{3q}{2}}\mathcal{E}[\dot{\Delta}_q \mathbf{U}](t)+2^{\frac{3q}{2}}\mathcal{D}[\dot{\Delta}_q \mathbf{U}](t)\\
\lesssim&\,2^{\frac{3q}{2}}\mathcal{E}[\dot{\Delta}_q \mathbf{U}](0)+\sqrt{\mathcal{E}[\mathbf{U}](t)+M(t)}2^{\frac{3q}{2}}\mathcal{D}[\dot{\Delta}_q \mathbf{U}](t)\\
&+\sqrt{c_q 2^{\frac{3q}{2}}\mathcal{E}[\dot{\Delta}_q \mathbf{U}](t)}\Vert \mathbf{U}\Vert_{\widetilde{L}^2_t(\dot{\mathbb{B}}_{2,1}^{1/2})}\\
&+\sqrt{M(t)}\Big(\sqrt{\Vert \mathbf{U}\Vert_{\widetilde{L}^2_t(\dot{\mathbb{B}}_{2,1}^{1/2})}+\Vert \mathbf{U}\Vert_{\widetilde{L}^2_t(\dot{\mathbb{B}}_{2,1}^{3/2})}}\Big)\sqrt{c_q 2^{\frac{3q}{2}}\mathcal{D}[\dot{\Delta}_q \mathbf{U}](t)},
\end{aligned}
\end{equation}
Exploiting  \eqref{Besov_Dya_Ineq}, we get 
\begin{equation}\label{Main_Freq_Energy_4}
\begin{aligned}
&2^{\frac{3q}{2}}\mathcal{E}[\dot{\Delta}_q \mathbf{U}](t)+2^{\frac{3q}{2}}\mathcal{D}[\dot{\Delta}_q \mathbf{U}](t)\\
\lesssim&\,c_q\Vert \mathbf{U}(0)\Vert_{\dot{\mathbf{\mathcal{B}}}_{2,1}^{3/2}}+c_q\sqrt{\mathcal{E}[\mathbf{U}](t)+M(t)}\Vert \mathbf{U}\Vert_{\widetilde{L}^2_t(\dot{\mathbb{B}}_{2,1}^{3/2})}\\
&+c_q\sqrt{\Vert \mathbf{U}\Vert_{\widetilde{L}_t^\infty(\dot{\mathbf{\mathcal{B}}}_{2,1}^{3/2})}}\Vert \mathbf{U}\Vert_{\widetilde{L}^2_t(\dot{\mathbb{B}}_{2,1}^{1/2})}\\
&+c_q\sqrt{M(t)}\Big(\sqrt{\Vert \mathbf{U}\Vert_{\widetilde{L}^2_t(\dot{\mathbb{B}}_{2,1}^{1/2})}+\Vert \mathbf{U}\Vert_{\widetilde{L}^2_t(\dot{\mathbb{B}}_{2,1}^{3/2})}}\Big)\sqrt{\Vert \mathbf{U}\Vert_{\widetilde{L}^2_t(\dot{\mathbb{B}}_{2,1}^{3/2})}}.
\end{aligned}
\end{equation}
This  implies 
\begin{equation}\label{Main_Freq_Energy_5}
\begin{aligned}
&2^{\frac{3q}{2}}\mathcal{E}[\dot{\Delta}_q \mathbf{U}](t)+2^{\frac{3q}{2}}\mathcal{D}[\dot{\Delta}_q \mathbf{U}](t)\\
\lesssim&\,c_q\Vert \mathbf{U}(0)\Vert_{\dot{\mathbf{\mathcal{B}}}_{2,1}^{3/2}}+c_q\sqrt{\mathcal{E}[\mathbf{U}](t)+M(t)+\Vert \mathbf{U}\Vert_{\widetilde{L}_t^\infty(\dot{\mathbf{\mathcal{B}}}_{2,1}^{3/2})}}\\
&\qquad\qquad\qquad\qquad\times\Big(\Vert \mathbf{U}\Vert_{\widetilde{L}^2_t(\dot{\mathbb{B}}_{2,1}^{1/2})}+\Vert \mathbf{U}\Vert_{\widetilde{L}^2_t(\dot{\mathbb{B}}_{2,1}^{3/2})}\Big). 
\end{aligned}
\end{equation}
Hence, summing up on $q\in \Z$, we arrive at 
\begin{equation}\label{Main_Homogenous_Estimate}
\begin{aligned}
&\Vert \mathbf{U}\Vert_{\widetilde{L}_t^\infty(\dot{\mathbf{\mathcal{B}}}_{2,1}^{3/2})}+\Vert \mathbf{U}\Vert_{\widetilde{L}^2_t(\dot{\mathbb{B}}_{2,1}^{3/2})}\\
\lesssim&\,\Vert\mathbf{U}(0)\Vert_{\dot{\mathbf{\mathcal{B}}}_{2,1}^{3/2}}+\sqrt{\mathcal{E}[\mathbf{U}](t)+M(t)+\Vert \mathbf{U}\Vert_{\widetilde{L}_t^\infty(\dot{\mathbf{\mathcal{B}}}_{2,1}^{3/2})}}\\
&\qquad \qquad \qquad\times\Big(\Vert \mathbf{U}\Vert_{\widetilde{L}^2_t(\dot{\mathbb{B}}_{2,1}^{1/2})}+\Vert \mathbf{U}\Vert_{\widetilde{L}^2_t(\dot{\mathbb{B}}_{2,1}^{3/2})}\Big). 
\end{aligned}
\end{equation}
Combining \eqref{Main_Estimate_Order_0} and \eqref{Main_Homogenous_Estimate} and using the fact that $$M_0(t)+M(t)\lesssim\Vert \mathbf{U}(t)\Vert_{\widetilde{L}_t^\infty(\mathbf{\mathcal{B}}_{2,1}^{3/2})} ,$$ we conclude that \eqref{Main_inHomogenous_Estimate_1} is satisfied. This finishes the proof of Proposition \ref{Main_Proposition}
\end{proof}
\subsection{Proof of Theorem \ref{Main_Theorem_Nonl}}\label{Proof_Theorem_1}
We put 
\begin{equation}
\mathrm{Y}(t):=\Vert \mathbf{U}(t)\Vert_{\widetilde{L}_t^\infty(\mathbf{\mathcal{B}}_{2,1}^{3/2})}+\Vert \mathbf{U}(t)\Vert_{\widetilde{L}^2_t(\mathbb{B}_{2,1}^{3/2})}.  
\end{equation}
 It is clear that $\widetilde{L}_t^\infty(\mathbf{\mathcal{B}}_{2,1}^{s_2})\hookrightarrow \widetilde{L}_t^\infty(\mathbf{\mathcal{B}}_{2,1}^{s_1})$ and $\widetilde{L}^2_t(\mathbb{B}_{2,1}^{s_2})\hookrightarrow \widetilde{L}^2_t(\mathbb{B}_{2,1}^{s_1})$ for $s_2\geq s_1\geq 0.$

From, \eqref{Main_inHomogenous_Estimate_1} yields
\begin{eqnarray}\label{Main_Y_Estimate}
\mathrm{Y}(t)\leq C\mathrm{Y}(0)+C\mathrm{Y}^{3/2}(t),
\end{eqnarray}
where $C$ is a positive constant that does not depend on $t$. 
From \eqref{Main_Y_Estimate}, we conclude in a standard way (see Lemma \ref{Lemma_Stauss}) that there is $\delta>0$ small enough, such that if $\mathrm{Y}(0)=\Vert\mathbf{ U}_0\Vert_{\mathbf{\mathcal{B}}_{2,1}^{3/2}}\leq \delta $, then there is $K>0$, independent of time,  such that
\begin{eqnarray*}
\mathrm{Y}(t)\leq K.
\end{eqnarray*}
 This uniform estimate allows to continue the local solution to be global in time.
 \section{Decay estimates of the linear  equation}\label{Decay_Estimate_Linearized_Model}
 In this section, we show  decay estimates for some energy norms for  the solution of 
 linearized system 
 \begin{subequations}
\label{Main_problem_Linearized}
\begin{eqnarray}
\tau u_{ttt}+ u_{tt}-\Delta u-\beta \Delta u_{t}=0, \qquad x\in \R^n,\, t>0, \label{MGT_1_Linearized}
\end{eqnarray}
with the initial conditions:
\begin{eqnarray}
u(t=0)=u_{0},\qquad u_{t}(t=0)=u_{1}\qquad u_{tt}(t=0)=u_{2},
\label{Initial_Condition_Linearized}
\end{eqnarray}
\end{subequations}
  by assuming additionally that our initial data are  in  $\dot{B}^{-\rho}_{p,\infty}$ with  $1-\frac{1}{p}=\frac{\rho}{n}$.  
 This improves the decay estimates given in \cite{PellSaid_2019_1},  due to the embedding  $L^1\hookrightarrow \dot{B}^{-\rho}_{p,\infty}$ (see Lemma \ref{Embedding_Lemma}). These decay estimates of the linearized problem are crucial in the proof of the decay estimates of the nonlinear problem stated in Theorem \ref{Decay_Estimate_Theorem}. 
 We define the vector  $\mathbf{V}=(u_t+\tau u_{tt}, \nabla (u+\tau u_t), \nabla u_t)$ and   
 show the following   decay estimates for the vector solution $\mathbf{V}$, where here we are not restricted to the 3D case.  
 
 \begin{theorem}\label{Theorem_Decay}
 Assume that $0<\tau < \beta$. Suppose that  $\mathbf{V}_0\in \dot{B}_{2,1}^\sigma(\R^n)\cap \dot{B}_{2,\infty}^{-s}(\R^n)$ for $\sigma\in \R$, $s\in \R$ satisfying $\sigma+s>0$, then  for all $t\geq 0$, the solution of \eqref{Main_problem_Linearized} has the following decay estimate:  
 \begin{subequations}\label{Decay_Estimates_V}
 \begin{equation}\label{Decay_Homog}
\begin{aligned}
\Vert \mathbf{V} (t)\Vert_{\dot{B}_{2,1}^{\sigma}}\lesssim \Vert \mathbf{V}_0\Vert_{\dot{B}_{2,1}^\sigma\cap \dot{B}_{2,\infty}^{-s}}(1+t)^{-\frac{\sigma+s}{2%
}},
\end{aligned}
\end{equation}
Also if $\mathbf{V}_0\in \dot{B}_{2,1}^\sigma(\R^n)\cap \dot{B}_{2,\infty}^{-s}(\R^n)$ for $\sigma\geq 0$ and $s>0$, then the solution of \eqref{Main_problem_Linearized} has the following decay estimate: 
\begin{equation}\label{Decay_InHomog}
\begin{aligned}
\Vert\Lambda^k \mathbf{V} (t)\Vert_{B_{2,1}^{\sigma-k}}\lesssim \Vert \mathbf{V}_0\Vert_{\dot{B}_{2,1}^\sigma\cap \dot{B}_{2,\infty}^{-s}}(1+t)^{-\frac{k+s}{2%
}}
\end{aligned}
\end{equation}
for $0\leq k\leq \sigma$. 
\end{subequations}
 \end{theorem}
 The result of Theorem \ref{Theorem_Decay} provides the following corollary. 
 \begin{corollary}\label{Decay_Estimate_L_p}
 Under the assumptions of Theorem \ref{Theorem_Decay}, it holds that for $2\leq p<\infty$ and $1\leq q<p$, 
 \begin{equation}\label{L_p_Decay}
\Vert \mathbf{V}(t)\Vert_{L^p}\lesssim \Vert \mathbf{V}_0\Vert_{\dot{B}^{n(\frac{1}{2}-\frac{1}{p})}_{2,1}\cap \dot{B}_{2,\infty}^{-n(\frac{1}{q}-\frac{1}{2})}}(1+t)^{-\frac{n}{2%
}(\frac{1}{q}-\frac{1}{p})}. 
\end{equation}
 \end{corollary}
 \begin{proof}[Proof of Corollary \ref{Decay_Estimate_L_p}] Using for $(2\leq p<\infty)$,  the estimate 
\begin{equation}\label{L_p_Estimate}
\Vert \mathbf{V}\Vert_{L^p}\lesssim \Vert \mathbf{V}\Vert_{\dot{H}^{n(\frac{1}{2}-\frac{1}{p})}}\lesssim \Vert \mathbf{V}\Vert_{\dot{B}^{n(\frac{1}{2}-\frac{1}{p})}_{2,1}}
\end{equation}
we obtain from \eqref{Decay_Homog} and for $s=n(\frac{1}{q}-\frac{1}{2}), \, (1\leq q<p)$,
the estimate  \eqref{L_p_Decay}. 
\end{proof}

\begin{remark}
The decay rate \eqref{L_p_Decay} has been proven  in \cite{PellSaid_2019_1} for $p=2$ and for initial data in $L^1(\R^n)\cap L^2(\R^n)$. Comparing  \eqref{L_p_Decay} to  the decay estimates in \cite{PellSaid_2019_1} we can say that  the assumption $\mathbf{V}_0\in L^1(\R^n)$ is relaxed due to the embedding $L^1(\mathbb{R}%
^n)\hookrightarrow \dot{{B}}_{2,\infty}^{-n/2}(\mathbb{R}^n)$ (see Lemma \ref{Embedding_Lemma}). However the assumption $\mathbf{V}_0\in \dot{B}^{0}_{2,1}(\R^n)$ is more restrictive than $\mathbf{V}_0\in L^2(\R^n)$ due to the embedding $\dot{B}^{0}_{2,1}(\R^n)\hookrightarrow \dot{B}^{0}_{2,2}(\R^n)=L^2(\R^n)$. 
\end{remark}

The following proposition has been proved in \cite[Proposition 3.1]{PellSaid_2019_1}. 
\begin{proposition}\label{Proposition_Linearized}
Assume that $0<\tau < \beta$. 
There exists a functional $\mathcal{L}[\hat{\mathbf{U}}]$ and three positive constants $c_1, c_2$ and $\eta$ such that  for all $t\geq 0$, it holds that 
\begin{equation}
\begin{aligned}
c_1\mathrm{E}_1 [\hat{\mathbf{U}}](t) \leq \mathcal{L}[\hat{\mathbf{U}}](t)\leq c_2 \mathrm{E}_1 [\hat{\mathbf{U}}](t),
\end{aligned}
\end{equation}
and 
\begin{equation}
\frac{\textup{d}}{\textup{d} t}\mathcal{L}[\hat{\mathbf{U}}](t)+\eta \frac{|\xi|^2}{1+|\xi|^2}\mathcal{L}[\hat{\mathbf{U}}](t)\leq 0,  
\end{equation}
where 
\begin{equation}
\begin{aligned}
\mathrm{E}_1 [\hat{\mathbf{U}}](t)=&\,\frac{1}{2} \Big\{|\hat{v}+\tau \hat{w}|^2+\tau (\beta-\tau) |\xi|^2 |\hat{v}|^2+|\xi|^2 |\hat{u}+\tau \hat{v}| ^2\Big\}. \\
\approx &\, |\hat{\mathbf{V}}(t)|^2.
\end{aligned}
\end{equation}

\end{proposition}  

Applying $\dot{\Delta}_q,\,q\in \Z$ to \eqref{Main_problem_Linearized} and following the same steps as in the proof of the Proposition 3.1 in \cite{PellSaid_2019_1}, we obtain the frequency-localized estimate stated in the following proposition. 
 \begin{proposition}\label{Proposition_Homo_Local}
 Assume that $0<\tau < \beta$. 
There exists a functional $\mathcal{L}\big[\widehat{\dot{\Delta}_q\mathbf{U}}\big]$ and three positive constants $c_1, c_2$ and $\eta$ such that  for all $t\geq 0$, it holds that 
\begin{equation}
\begin{aligned}
c_1\mathrm{E}_1 \big[\widehat{\dot{\Delta}_q\mathbf{U}}\big](t) \leq \mathcal{L}\big[\widehat{\dot{\Delta}_q\mathbf{U}}\big](t)\leq c_2 \mathrm{E}_1 \big[\widehat{\dot{\Delta}_q\mathbf{U}}\big](t),
\end{aligned}
\end{equation}
and 
\begin{equation}\label{Lyap_Frequncy_Local}
\frac{\textup{d}}{\textup{d} t}\mathcal{L}\big[\widehat{\dot{\Delta}_q\mathbf{U}}\big](t)+\eta \frac{|\xi|^2}{1+|\xi|^2}\mathcal{L}\big[\widehat{\dot{\Delta}_q\mathbf{U}}\big](t)\leq 0.  
\end{equation}

 \end{proposition}
 Similar to Proposition \ref{Proposition_Homo_Local}, applying the frequency localized operator $\Delta_q, \, q\geq -1$ to system \eqref{Main_problem_Linearized} and using the same method in \cite{PellSaid_2019_1}, we can prove the following proposition. 
 
 \begin{proposition}\label{Proposition_inHomo_Local}
 Assume that $0<\tau < \beta$. 
There exists a functional $\mathcal{L}\big[\widehat{\Delta_q\mathbf{U}}\big]$ and three positive constants $c_1, c_2$ and $\eta$ such that  for all $t\geq 0$, it holds that for all $q\geq -1$ (with $\Delta_{-1}=S_0$)
\begin{equation}
\begin{aligned}
c_1\mathrm{E}_1 \big[\widehat{\Delta_q\mathbf{U}}\big](t) \leq \mathcal{L}\big[\widehat{\Delta_q\mathbf{U}}\big](t)\leq c_2 \mathrm{E}_1 \big[\widehat{\Delta_q\mathbf{U}}\big](t),
\end{aligned}
\end{equation}
and 
\begin{equation}\label{Lyap_Frequncy_Local_Inhom}
\frac{\textup{d}}{\textup{d} t}\mathcal{L}\big[\widehat{\Delta_q\mathbf{U}}\big](t)+\eta \frac{|\xi|^2}{1+|\xi|^2}\mathcal{L}\big[\widehat{\Delta_q\mathbf{U}}\big](t)\leq 0.  
\end{equation}

 \end{proposition}
 
 \subsection{Proof of Theorem \ref{Theorem_Decay}}
 
 Since $\mathcal{L}\big[\widehat{\dot{\Delta}_q\mathbf{U}}\big]\approx \mathrm{E}_1 \big[\widehat{\dot{\Delta}_q\mathbf{U}}\big]\approx  |\widehat{\dot{\Delta}_q\mathbf{V}}|^2$, then, it holds from \eqref{Lyap_Frequncy_Local} that 
  \begin{equation}\label{Localized_Estimate}
|\widehat{\dot{\Delta}_q\mathbf{V}}|^2\lesssim e^{-\eta\frac{|\xi|^2}{1+|\xi|^2}t}|\widehat{\dot{\Delta}_q\mathbf{V}_0}|^2. 
\end{equation} 
Following \cite{Xu_Kawashima_2015} and  using the definition of the localization operator $\dot{\Delta}_q$, we have $|\xi|\sim 2^q$
and thus, for $q<0$, then $|\xi|<1$ and for $q\geq 0$, then $|\xi|\geq 1$.
Hence, we separate the proof into two frequency regions, low frequency $%
|\xi|<1$ and high frequency $|\xi|\geq 1$.

\subsubsection*{Case 1.} For $q\geq 0$, multiplying  \eqref{Localized_Estimate} by $|\xi|^k$  and by using Plancherel's theorem, we deduce that there exists a constant $c>0$ such that 
\begin{equation}\label{High_Frequency_V}
\Vert \dot{\Delta}_q \Lambda ^k\mathbf{V}\Vert_{L^2}\lesssim e^{-ct} \Vert \dot{\Delta}_q \Lambda ^k\mathbf{V}_0\Vert_{L^2}. 
\end{equation}
We multiply the above inequality by $2^{q(\sigma-k)}$, summing the resulting inequality over $q\geq 0$ and using Lemma \ref{Lemma_Bernstein},   we get
\begin{equation}\label{High_Frequency_Estimate_Inhom}
\begin{aligned}
\sum_{q\geq 0} 2^{q(\sigma-k)}\Vert \dot{\Delta}_q \Lambda ^k\mathbf{V}\Vert_{L^2}\lesssim &\,e^{-ct}\sum_{q\geq 0}2^{q(\sigma-k)}\Vert \dot{\Delta}_q \Lambda ^k\mathbf{V}_0\Vert_{L^2} \\
\lesssim &\,e^{-ct} \Vert \mathbf{V}_0\Vert_{\dot{B}_{2,1}^\sigma}. 
\end{aligned}
\end{equation}

\subsubsection*{Case 2.} For $q<0$, we have $\frac{|\xi|^2}{1+|\xi|^2}\approx  |\xi|^2$ and $|\xi|\sim 2^{q}$. Thus, we have, for some $\eta_1>0$, 
\begin{equation}\label{Localized_Estimate_2}
|\widehat{\dot{\Delta}_q\mathbf{V}}|\lesssim e^{-\eta_1|\xi|^2t}|\widehat{\dot{\Delta}_q\mathbf{V}_0}|\approx e^{-\eta_12^{2q}t}|\widehat{\dot{\Delta}_q\mathbf{V}_0}|.  
\end{equation}
Hence, integrating \eqref{Localized_Estimate_2} over $\R^3_\xi$ and using  Plancherel's theorem, we obtain 
\begin{equation}
\Vert \dot{\Delta}_q\mathbf{V}\Vert_{L^2}\lesssim e^{-\eta_1 (2^q \sqrt{t})^2} \Vert \dot{\Delta}_q\mathbf{V}_0\Vert_{L^2}. 
\end{equation}
Let $\sigma\in\mathbb{R}$ and $s\in \mathbb{R}$ such that $\sigma+s>0$, then
we multiply the above inequality by $2^{q\sigma}$, we get 
\begin{eqnarray}  \label{Estimate_1}
 2^{q\sigma}\Vert \dot{\Delta}_q\mathbf{V}\Vert_{L^2} &\lesssim &\Vert \mathbf{V}_0 \Vert_{\dot{%
B}_{2,\infty}^{-s}}(1+t)^{-\frac{\sigma+s}{2}} \Big[(2^q\sqrt{t})^{\sigma+s}
e^{-\eta_1(2^q\sqrt{t})^2}\Big]  
\end{eqnarray}
since $ \sup_{t>0}\sum_{q\in \Z}(2^q\sqrt{t})^{\sigma+s} e^{-c(2^q\sqrt{t})^2}\lesssim  1$ (see \cite[Lemma 2.35]{bahouri2011fourier}), 
then we obtain 
\begin{equation}\label{low_Frequncy_Estimate}
\sum_{q<0}2^{q\sigma}\Vert \dot{\Delta}_q\mathbf{V}\Vert_{L^2}\lesssim 
 \Vert \mathbf{V}_0 \Vert_{\dot{B}_{2,\infty}^{-s}}(1+t)^{-\frac{\sigma+s}{2%
}}.
\end{equation}
Consequently, collecting \eqref{High_Frequency_Estimate_Inhom} and \eqref{low_Frequncy_Estimate}, we obtain 
\begin{equation}
\begin{aligned}
\Vert \mathbf{V} \Vert_{\dot{B}_{2,1}^{\sigma}}=&\,\sum_{q<0}2^{q\sigma}\Vert \dot{\Delta}_q\mathbf{V}\Vert_{L^2} +\sum_{q\geq 0} 2^{q\sigma}\Vert \dot{\Delta}_q \mathbf{V}\Vert_{L^2}\\
\lesssim&\, \Vert \mathbf{V}_0 \Vert_{\dot{B}_{2,\infty}^{-s}}(1+t)^{-\frac{\sigma+s}{2%
}}+e^{-ct} \Vert \mathbf{V}_0\Vert_{\dot{B}_{2,1}^\sigma}\\
\lesssim&\, \Vert \mathbf{V}_0\Vert_{\dot{B}_{2,1}^\sigma\cap \dot{B}_{2,\infty}^{-s}}(1+t)^{-\frac{\sigma+s}{2%
}},
\end{aligned}
\end{equation}
which is \eqref{Decay_Homog}. 

Now, we need to show \eqref{Decay_InHomog}. 
Since for $q\geq 0$ we have $\dot{\Delta}_qf=\Delta_qf$, then the high-frequency estimate holds as in the homogeneous case,  that is 
\begin{equation}\label{High-Frequncy_Estimate}
\sum_{q\geq 0} 2^{q(\sigma-k)}\Vert \Delta_q \Lambda ^k\mathbf{V}\Vert_{L^2}\lesssim  e^{-ct} \Vert \mathbf{V}_0\Vert_{\dot{B}_{2,1}^\sigma}. 
\end{equation}
To show the low-frequency estimate, we take $q=-1$ 
in \eqref{Lyap_Frequncy_Local_Inhom}, we obtain 
\begin{equation}\label{Lyap_Frequncy_Local_Inhom_2}
\frac{\textup{d}}{\textup{d} t}\mathcal{L}\big[\widehat{\Delta_{-1}\mathbf{U}}\big](t)
+\eta_2 |\xi|^2 \vert \widehat{\Delta_{-1}\mathbf{V}}\vert_{L^2}^2\leq 0. 
\end{equation}
 Multiplying  \eqref{Lyap_Frequncy_Local_Inhom_2} by $|\xi|^{2k}$ and  using  Plancherel's theorem,  
 it follows that 
\begin{equation}\label{Ineq_L_Main}
\frac{\textup{d}}{\textup{d} t}\mathscr{L}\big[\Delta_{-1}\mathbf{U}\big](t) +\eta_2 \Vert \Lambda^{k+1}\Delta_{-1}\mathbf{V}\Vert_{L^2}^2\leq 0
\end{equation}
where 
\begin{equation}\label{Equiva_L_S_0}
\mathscr{L}\big[\Delta_{-1}\mathbf{U}\big](t)=\int_{\R^3_\xi}|\xi|^{2k}\mathcal{L}\big[\widehat{\Delta_{-1}\mathbf{U}}\big]d\xi\approx \Vert \Lambda^{k} \Delta_{-1}\mathbf{V} (t)\Vert_{L^2}^2 . 
\end{equation}
Following \cite{Guo_Wang_2012} and applying \eqref{Interpolation_I},   we obtain 
\begin{equation}\label{Interpolation_II}
\begin{aligned}
\Vert \Lambda^{k+1}\Delta_{-1}\mathbf{V}\Vert_{L^2}\gtrsim \Vert \Lambda^{k} \Delta_{-1}\mathbf{V} \Vert_{L^2}^{1+\frac{1}{k+s}}\Vert \mathbf{V}\Vert_{\dot{B}_{2,\infty}^{-s}}^{-\frac{1}{k+s}}.
\end{aligned}
\end{equation}
Our next goal is to show that 
\begin{equation}\label{Negative_Sobolev_Inequality}
\Vert \mathbf{V}\Vert_{\dot{B}_{2,\infty}^{-s}}\lesssim \Vert \mathbf{V}_0\Vert_{\dot{B}_{2,\infty}^{-s}}. 
\end{equation}
 Indeed, using the fact that 
$\mathcal{L}\big[\widehat{\dot{\Delta}_q\mathbf{U}}\big]\approx |\widehat{\dot{\Delta}_q\mathbf{V} }|^2$, we obtain from 
\eqref{Lyap_Frequncy_Local} together with Plancherel's theorem   
\begin{equation}\label{L_2_Diady}
\Vert \dot{\Delta}_q\mathbf{V} (t)\Vert_{L^2} \lesssim \Vert \dot{\Delta}_q\mathbf{V}_0\Vert_{L^2}.
\end{equation}
Multiplying \eqref{L_2_Diady} by $2^{-qs}$ and taking the $\ell^\infty$ norm yields \eqref{Negative_Sobolev_Inequality}.

\noindent Now, combining \eqref{Ineq_L_Main}, \eqref{Interpolation_II} and \eqref{Negative_Sobolev_Inequality}, we obtain 
\begin{equation}\label{Ineq_L_Main_2}
\begin{aligned}
\frac{\textup{d}}{\textup{d} t}\mathscr{L}\big[\Delta_{-1}\mathbf{U}\big](t) +C (\Vert \Lambda^{k} \Delta_{-1}\mathbf{V} \Vert_{L^2}^2)^{1+\frac{1}{k+s}}\Vert \mathbf{V}_0\Vert_{\dot{B}_{2,\infty}^{-s}}^{-\frac{2}{k+s}}\leq 0. 
\end{aligned}
\end{equation}
Recalling \eqref{Equiva_L_S_0}, we have from \eqref{Ineq_L_Main_2} for some $C_0>0$, 
\begin{equation}
\begin{aligned}
\frac{\textup{d}}{\textup{d} t}\mathscr{L}\big[\Delta_{-1}\mathbf{U}\big](t) +C_0 (\mathscr{L}\big[\Delta_{-1}\mathbf{U}\big](t))^{1+\frac{1}{k+s}}\Vert \mathbf{V}_0\Vert_{\dot{B}_{2,\infty}^{-s}}^{-\frac{2}{k+s}}\leq 0,
\end{aligned}
\end{equation}
which implies  
\begin{equation}
\mathscr{L}\big[\Delta_{-1}\mathbf{U}\big](t)\leq \Big(\left[\mathscr{L}\big[\Delta_{-1}\mathbf{U}\big](0)\right]^{-\frac{1}{k+s}}+C_0 \Vert \mathbf{V}_0\Vert_{\dot{B}_{2,\infty}^{-s}}^{-\frac{2}{k+s}} \frac{t}{k+s}\Big)^{-k+s}. 
\end{equation}
Using again \eqref{Equiva_L_S_0}, we obtain 
\begin{equation}\label{Estim_Nonhom_S_0}
\begin{aligned}
\Vert \Lambda^{k} \Delta_{-1}\mathbf{V} (t)\Vert_{L^2}^2\lesssim&\, \Big(\left[\Vert \Lambda^{k} \Delta_{-1}\mathbf{V} (0)\Vert_{L^2}\right]^{-\frac{2}{k+s}}+C_0 \Vert \mathbf{V}_0\Vert_{\dot{B}_{2,\infty}^{-s}}^{-\frac{2}{k+s}} \frac{t}{k+s}\Big)^{-k+s}\\
\end{aligned}
\end{equation}
 Then, we have by taking the square root of \eqref{Estim_Nonhom_S_0} (see \cite{Wang_Xu_Kawashima_2020}) 
%

  \begin{equation}\label{Low_Frequency_Decay}
\Vert \Lambda^{k} \Delta_{-1}\mathbf{V} \Vert_{L^2}\lesssim \Vert \mathbf{V}_0\Vert_{\dot{B}_{2,\infty}^{-s}} (1+t)^{-\frac{k+s}{2}}. 
\end{equation}
Collecting  \eqref{High_Frequency_Estimate_Inhom} and \eqref{Low_Frequency_Decay}, we obtain \eqref{Decay_InHomog}. This finishes the proof of Theorem \ref{Theorem_Decay}. 

\subsection{Decay estimates for $u_t$}
Theorem \ref{Theorem_Decay} does not directly yield a decay rate for $\Vert u_t\Vert_{{X_{2,1}^\sigma}}=\Vert  v\Vert_{{X_{2,1}^\sigma}}$ (where  ${X_{2,1}^\sigma}$ is either $\dot{B}_{2,1}(\R^n)$ or $B_{2,1}(\R^n))$. Such a decay rate is necessary  to prove the decay estimates of the nonlinear equation. 
However, we can obtain it through the bound
   \begin{equation}\label{Ineq_L_2}
\Vert  v\Vert_{{X_{2,1}^\sigma}}\lesssim \Vert (v+\tau w)\Vert_{{X_{2,1}^\sigma}}+\Vert  w\Vert_{{X_{2,1}^\sigma}}
\end{equation} 
and \eqref{Decay_Estimates_V}, if we have a decay rate for  $\Vert w\Vert_{{X_{2,1}^\sigma}}$. The decay  rate of $\Vert  w\Vert_{{X_{2,1}^\sigma}}
$ is the result of the next proposition. 
\begin{proposition}\label{Decay_w_New}
Assume that $0<\tau < \beta$. If $\mathbf{V}_0\in \dot{B}_{2,1}^{\sigma+1}(\R^n)\cap \dot{B}_{2,\infty}^{-s}(\R^n)$ and $w_0\in \dot{B}_{2,1}^\sigma(\R^n)$ for $\sigma\in \R$, $s\in \R$ satisfying $\sigma+s+1>0$, then it holds that 
\begin{equation}\label{Decay_estimate_w_homog}
\begin{aligned}
\Vert w (t)\Vert_{\dot{B}_{2,1}^{\sigma}}\lesssim \Big(\Vert \mathbf{V}_0\Vert_{\dot{B}_{2,1}^{\sigma+1}\cap \dot{B}_{2,\infty}^{-s}}+\Vert w_0\Vert_{\dot{B}_{2,1}^\sigma}\Big)(1+t)^{-\frac{\sigma+s+1}{2%
}}
\end{aligned}
\end{equation}
provided that the thermal relaxation $\tau>0$ is sufficiently small.
\end{proposition}   

 \begin{proof}
For proving the above estimate, we need to employ the decay rates of the Fourier transform of the solution.  It has been proved in \cite[Estimate 3.6]{PellSaid_2019_1} that the following estimate holds:
 \begin{equation}\label{Eexp}
     |\hat{\mathbf{V}}(\xi,t)|^2\lesssim |\hat{\mathbf{V}}(\xi,0)|^2\exp{(-\lambda \tfrac{|\xi|^2}{1+|\xi|^2}t)}
     \end{equation} 
for all $t \geq 0$. The constant $\lambda$ is positive and independent of $t$ and $\xi$. For the linearized problem, it  holds from the third equation of \eqref{System_New} 
\begin{equation} 
\begin{aligned}
\frac{1}{2}\frac{\textup{d}}{\dt}\int_{\mathbb{R}^{n}}\tau \left\vert  
w\right\vert ^{2}\dx+\int_{\mathbb{R}^{n}} |w|^2 \dx
=&\,
\int_{\mathbb{R}^{n}} \Delta uw\dx+\beta\int_{\mathbb{R}^{n}} \Delta vw\dx\\
\leq &\,\frac{1}{4}\Vert w\Vert_{L^2}^2+\Vert \Delta u\Vert_{L^2}^2+\frac{1}{4}\Vert w\Vert_{L^2}^2+\beta^2 \Vert \Delta v\Vert_{L^2}^2. 
\end{aligned}
\end{equation}
Hence, this yields 

	\begin{equation} 
\begin{aligned}
\frac{1}{2}\frac{\textup{d}}{\dt}\int_{\mathbb{R}^{n}}\tau \left\vert  
w\right\vert ^{2}\dx+\frac{1}{2}\int_{\mathbb{R}^{n}} |w|^2 \dx
\lesssim \,
\Vert
\Delta (u+\tau v)\Vert_{L^2}^2+\Vert \Delta v\Vert _{L^{2}}^2.
\end{aligned}
\end{equation}
 Thus we know that
  \begin{equation}\label{w_Energy_Fourier}
	\begin{aligned}
	\frac{1}{2}\frac{\textup{d}}{\dt}\tau \left\vert \hat{w}\right\vert
	^{2}+ \frac{1}{2}|\hat{w}|^2 \lesssim \, |\xi|^2 |\hat{\mathbf{V}}(\xi,t)|^2.
	\end{aligned}
	\end{equation}
By plugging in estimate \eqref{Eexp} for $|\hat{\mathbf{V}}(\xi,t)|^2$ in the above inequality, we obtain
\begin{equation}
\frac{\textup{d}}{\dt} \left\vert \hat{w}\right\vert
	^{2}\leq -\frac{1}{\tau} \left\vert \hat{w}\right\vert
	^{2}+C|\xi|^2 |\hat{\mathbf{V}}(\xi,0)|^2\exp{\left(-\lambda \tfrac{|\xi|^2}{1+|\xi|^2}t\right)}.  
\end{equation} 
Using the factorization $\frac{\textup{d}}{\dt} \vert \hat{w}\vert
	^{2}+\frac{1}{\tau} \vert \hat{w}\vert
	^{2}=e^{-\frac{t}{\tau}}\frac{\textup{d}}{\dt} (e^{\frac{t}{\tau}}\vert \hat{w}\vert
	^{2} )$
to arrive at 
 \begin{equation}\label{w_Estimate_main}
 \begin{aligned}
\left\vert \hat{w}\right\vert
	^{2} \leq&\, |\hat{w}_0|^2 \exp{\left(-\tfrac{1}{\tau} t\right)} +C|\xi|^2|\hat{\mathbf{V}}(\xi,0)|^2\int_0^t \exp{\left(-\lambda \tfrac{|\xi|^2}{1+|\xi|^2} s\right)}\, \exp{\left(-\tfrac{1}{\tau}(t-s)\right)}\ds,
\end{aligned}	
\end{equation}
which directly leads to	
 \begin{equation}\label{w_Estimate_main}
\begin{aligned}
\left\vert \hat{w}\right\vert
^{2} \leq& \begin{multlined}[t]
	 \, |\hat{w}_0|^2 \exp{\left(-\tfrac{1}{\tau} t\right)}\\ +C|\xi|^2|\hat{\mathbf{V}}(\xi,0)|^2\exp{\left(-\tfrac{1}{\tau} t\right)}\left(\tfrac{1}{\tau}-\lambda \tfrac{|\xi|^2}{1+|\xi|^2}\right)^{-1} \left[\exp{\left(-(\lambda \tfrac{|\xi|^2}{1+|\xi|^2}-\tfrac{1}{\tau})t\right)} -1\right]. \end{multlined}
\end{aligned}	
\end{equation}
To further bound the right-hand side side, we can use the identity 
\[\frac{1}{\frac{1}{\tau}-\lambda \frac{|\xi|^2}{1+|\xi|^2}}=
\frac{\tau  \left(|\xi|^2+1\right)}{|\xi|^2 (1-\lambda  \tau )+1}.
\] 
Assuming that the thermal relaxation is small enough so that $\tau< \frac{1}{\lambda}$, it holds
\begin{equation}
\frac{1}{\frac{1}{\tau}-\lambda \frac{|\xi|^2}{1+|\xi|^2}}\leq \frac{\tau}{1-\lambda  \tau}.
\end{equation}
Altogether, for small $\tau>0$, we obtain 
\begin{equation}\label{w_Fourier_Estimate}
\left\vert \hat{w}\right\vert
	^{2}\leq |\hat{w}_0|^2 \exp{\left(-\tfrac{1}{\tau} t\right)}+C|\xi|^2 |\hat{\mathbf{V}}(\xi,0)|^2\exp{\left(-\lambda \tfrac{|\xi|^2}{1+|\xi|^2}t\right)}.  
\end{equation}
Using the third equation of the linearized system of \eqref{System_New_Localized} (i.e., $\mathrm{R}_q=0$) and following the above steps, we can prove the following estimate (here $|\xi|\approx 2^{q}$)
\begin{equation}\label{w_Fourier_Estimate_Local}
\vert \widehat{\dot{\Delta}_qw}\vert
	^{2}\leq |\widehat{\dot{\Delta}_qw}_0|^2 \exp{\left(-\tfrac{1}{\tau} t\right)}+C2^{2q} |\widehat{\dot{\Delta}_q\mathbf{V}_0}|^2\exp{\left(-\lambda \tfrac{|\xi|^2}{1+|\xi|^2}t\right)}.  
\end{equation}
As in the proof of Theorem \ref{Theorem_Decay}, for $q\geq 0$, we have from \eqref{w_Fourier_Estimate_Local} and Plancherel's theorem
 \begin{equation}
\Vert \dot{\Delta}_qw\Vert_{L^2}\lesssim e^{-ct}\Big(\Vert \dot{\Delta}_qw_0\Vert_{L^2}+ 2^q\Vert \dot{\Delta}_q\mathbf{V}_0\Vert_{L^2}\Big). 
\end{equation}
We multiply the above inequality by $2^{q\sigma}$, summing the resulting inequality over $q\geq 0$ and using Lemma \ref{Lemma_Bernstein},   we get
\begin{equation}\label{w_high_Frequency}
\begin{aligned}
\sum_{q\geq 0} 2^{q\sigma}\Vert \dot{\Delta}_qw\Vert_{L^2}\lesssim& \,e^{-ct}\sum_{q\geq 0}\Big(2^{q(\sigma+1)}\Vert \dot{\Delta}_q \mathbf{V}_0\Vert_{L^2}+2^{q\sigma}\Vert \dot{\Delta}_qw_0\Vert_{L^2}\Big)\\
\lesssim &\,e^{-ct} \big(\Vert \mathbf{V}_0\Vert_{\dot{B}_{2,1}^{\sigma+1}}+\Vert w_0\Vert_{\dot{B}_{2,1}^{\sigma}}\big). 
\end{aligned}
\end{equation}
Now, for $q<0$, we have as in  \eqref{Localized_Estimate_2}
\begin{equation}
\Vert \dot{\Delta}_qw\Vert_{L^2}\lesssim e^{-\frac{1}{2\tau} t}\Vert \dot{\Delta}_qw_0\Vert_{L^2}+  2^qe^{-\eta_1 (2^q \sqrt{t})^2} \Vert \dot{\Delta}_q\mathbf{V}_0\Vert_{L^2}. 
\end{equation}
As in the proof of Theorem \ref{Theorem_Decay}, for  $\sigma\in\mathbb{R}$ and $s\in \mathbb{R}$ such that $\sigma+s+1>0$, then
we multiply the above inequality by $2^{q\sigma}$, we get 
\begin{equation}\label{low_Frequncy_Estimate_w}
\sum_{q<0}2^{q\sigma}\Vert \dot{\Delta}_qw\Vert_{L^2}\lesssim e^{-\frac{1}{\tau} t}\Vert w_0\Vert_{\dot{B}_{2,1}^\sigma}+
 \Vert \mathbf{V}_0 \Vert_{\dot{B}_{2,\infty}^{-s}}(1+t)^{-\frac{\sigma+s +1}{2%
}}.
\end{equation}
Hence, \eqref{Decay_estimate_w_homog} holds by collecting \eqref{w_high_Frequency} and \eqref{low_Frequncy_Estimate_w}. 
\end{proof}
\section{Decay of the nonlinear model}\label{Section_6}
Our goal in this section is to extend the decay estimates  in Theorem \ref{Theorem_Decay} to the nonlinear model \eqref{Main_problem}. We follow  the method in  
\cite{Xu_Kawashima_2015}, which is based on a frequency-localization Duhamel principle.  

First,  taking the Fourier transform of  linearized system associated to   \eqref{System_New}, we obtain 
\begin{equation}\label{System_New_0}
\left\{
\begin{array}{ll}
\hat{u}_t=\hat{v},\vspace{0.2cm}\\
\hat{v}_t=\hat{w},\vspace{0.2cm}\\
\hat{w}_t=-\dfrac{|\xi|^2}{\tau }\hat{u}-\dfrac{\beta |\xi|^2}{\tau }\hat{v}-\dfrac{1}{\tau }\hat{w}.
\end{array}
\right.
\end{equation}
We can write the previous system in a matrix form as
\begin{subequations}\label{Linear_System_Fourier}
\begin{equation}\label{Matrix_Form}
\hat{\mathbf{U}}_t(\xi,t)=\Phi(\xi) \hat{\mathbf{U}}(\xi,t),
\end{equation}
with the initial data
\begin{eqnarray}\label{Initial_data}
\hat{\mathbf{U}}(\xi,0)=\hat{\mathbf{U}}_0(\xi),
\end{eqnarray}
\end{subequations}
where $\hat{\mathbf{U}}(\xi,t)=(\hat{u}(\xi,t),\hat{v}(\xi,t),\hat{w}(\xi,t))^T$ and
\begin{equation}\label{LA}
\Phi(\xi)= L+|\xi|^2 A =
\left(
\begin{array}{ccc}
0 & 1 & 0 \\
0 & 0 & 1 \\
0 & 0 & -\dfrac{1}{\tau }
\end{array}%
\right)
+ |\xi|^2
\left(
\begin{array}{ccc}
0 & 0 & 0 \\
0 & 0 & 0 \\
-\dfrac{1}{\tau } & -\dfrac{\beta }{\tau } & 0
\end{array}%
\right).
\end{equation}
We denote by $\mathcal{G}(t,x)$ the Green matrix associated with the linearized system  and defined through its Fourier transform as  
\begin{equation}
\widehat{\mathcal{G}f}(t,\xi)=e^{t\Phi(\xi)}\hat{f}(\xi).
\end{equation}
Hence, the solution $\mathbf{U}$ of the nonlinear problem \eqref{System_New} can be expressed, using Duhamel's formula, as 
\begin{equation}
\mathbf{U}=\mathbf{U}^{0}+\mathcal{N}(\mathbf{U}),
\end{equation}%
where $\mathbf{U}^{0}$ and $\mathcal{N}(\mathbf{U})$ are given by
\begin{equation}
\mathbf{U}^{0}=\mathcal{G}(t,x)\mathbf{U}_{0}\qquad \text{and}\qquad
\mathcal{N}(\mathbf{U})=\int_{0}^{t}\mathcal{G}(t-r,x)\mathcal{F}(\mathbf{%
U},\nabla \mathbf{U})(r) \textup{d}r .
\end{equation}%
$\mathbf{U}^{0}$ satisfies the linear equation
\begin{equation}
\mathbf{U}_{t}^{0}-\mathcal{A}\mathbf{U}^{0}=0,\qquad \mathbf{U}^{0}\left(
x\right) =\mathbf{U}_{0}\left( x\right) ,  \label{Linear_1}
\end{equation}%
and $\mathcal{N}(\mathbf{U})$ satisfies the nonlinear equation with zero
initial data, that\ is
\begin{equation}
\partial _{t}\mathcal{N}(\mathbf{U})-\mathcal{A} \mathcal{N}(\mathbf{U})=%
\mathcal{F}(\mathbf{U},\nabla \mathbf{U}),\qquad \mathcal{N}\left( \mathbf{U}%
\right) \left( x,0\right) =0,  \label{Nonlinear_1}
\end{equation}%
with \begin{equation}
\mathcal{A} \left(
\begin{array}{c}
u \\
v \\
w%
\end{array}%
\right)= \left(
\begin{array}{c}
v \\
w \\
\dfrac{1}{\tau}\Delta(u+\beta v)-\dfrac{1}{\tau}w%
\end{array}%
\right)
\end{equation}
and $\mathcal{F}$ is the nonlinear term
\begin{eqnarray*}
\mathcal{F}(\mathbf{U},\nabla \mathbf{U})= \left(
\begin{array}{c}
0 \\
0 \\
\dfrac{1}{\tau}\dfrac{B}{A}vw+\dfrac{2}{\tau}\nabla u\nabla v%
\end{array}%
\right).
\end{eqnarray*}

In the next lemma, we state  the frequency localization principle. 

\begin{lemma}\label{Lemma_Dumhamel}
Assume that $0<\tau<\beta$. Let $\mathbf{U}=(u,v,w)^T$ be the solution of \eqref{M_Main_System}. Then, it holds that
\begin{equation}
\Delta_q\Lambda^{\ell}\mathbf{U}(t,x)=\Delta_q \Lambda^\ell [\mathcal{G}(t,x)\mathbf{U}_{0}]+\int_0^t\Delta_q \Lambda^\ell[\mathcal{G}(t-r,x)\mathcal{F}(\mathbf{%
U},\nabla \mathbf{U})(r)] \textup{d}r
\end{equation}
 for $q\geq -1$ and $\ell\in \R$, and 
 \begin{equation}
\dot{\Delta}_q\Lambda^{\ell}\mathbf{U}(t,x)=\dot{\Delta}_q \Lambda^\ell [\mathcal{G}(t,x)\mathbf{U}_{0}]+\int_0^t\dot{\Delta}_q \Lambda^\ell[\mathcal{G}(t-r,x)\mathcal{F}(\mathbf{%
U},\nabla \mathbf{U})(r)] \textup{d}r,
\end{equation}
for $q\in \Z$ and $\ell\in \R$. 
\end{lemma}

The proof of Lemma \ref{Lemma_Dumhamel} follows as in \cite[Lemma 5.1]{Xu_Kawashima_2015}. We omit the details. 

\noindent We define
\begin{equation}
\vecc{E}(t)=\sup_{0\leq \sigma\leq t}\Vert \mathbf{ U}(\sigma)\Vert_{\mathbf{\mathcal{B}}_{2,1}^{3/2}}
\end{equation}
 and inspired by \eqref{Decay_InHomog} and \eqref{Decay_w_New}, we define 
\begin{equation}
\begin{aligned}
\mathcal{M}(t):=&\,\sup_{0\leq \ell<3/2}\sup_{0\leq r\leq t}(1+r)^{\frac{\ell}{2%
}+\frac{3}{4}}\Vert\Lambda^\ell \mathbf{V} (r)\Vert_{B_{2,1}^{\frac{3}{2}-\ell}}+\sup_{0\leq r\leq t}(1+r)^{\frac{3}{2}}\Vert\Lambda^{\frac{3}{2}}\mathbf{V}(r)\Vert_{\dot{B}^{0}_{2,1}}\\
&+ \sup_{0\leq r\leq t}(1+r)^{\frac{3}{2}}\Vert w(r)\Vert_{\dot{B}^{3/2}_{2,1}}+\sup_{0\leq r\leq t} (1+r)^{\frac{3}{4}}\Vert w(r)\Vert_{\dot{B}^{0}_{2,1}} +\sup_{0\leq r\leq t}(1+r)^{\frac{3}{2}}\Vert\mathbf{V}(r)\Vert_{\dot{B}^{\frac{5}{2}}_{2,1}}.
\end{aligned}
\end{equation}
We point out here that $\Vert\mathbf{V}(r)\Vert_{\dot{B}^{\frac{5}{2}}_{2,1}}$ should be expected to decay at the rate $(1+r)^{-2}$, but it seems not clear how to reach this for the nonlinear problem. However,  the decay of rate  $(1+r)^{-3/2}$ is enough to prove our result.

Now, we define the norm 
\begin{equation}
\begin{aligned}
\left\Vert \mathbf{U}\right\Vert _{\mathcal{H}}^{2}  = &\, \Vert
\nabla v\Vert _{L^{2}(\mathbb{R}^{3})}^{2}+\Vert \nabla (u+\tau v)\Vert
_{L^{2}(\mathbb{R}^{3})}^{2} +\Vert v+\tau w\Vert _{L^{2}(\mathbb{R}%
^{3})}^{2}. \\
=&\, \Vert\mathbf{V}\Vert_{L^2}^2.
\end{aligned}
\label{norm_H}
\end{equation}
The  goal is next to prove that $\mathcal{M}(t)$ is uniformly bounded for all time if $ \mathbbmss{E}_0$ (see \eqref{E_0_ss}) is sufficiently small. The main step towards this goal is to prove the following  proposition. 
\begin{proposition}\label{Proposition_Decay_Nonl}
Let $\mathbf{U}(t,x)$ be the global solution given in Theorem \ref{Main_Theorem_Nonl}. Assume that $\mathbf{U}_0\in \mathbf{\mathcal{B}}_{2,1}^{3/2} $ and $\mathbf{V}_0\in  \dot{B}_{2,\infty}^{-3/2}(\R^3) $. Then, it holds that 
\begin{equation}\label{Estimate_M_Nonl}
\mathcal{M}(t)\lesssim  \mathbbmss{E}_0+\vecc{E}(t)\mathcal{M}(t)+\mathcal{M}^2(t),
\end{equation}
where we recall that $\mathbbmss{E}_0=\Vert \mathbf{ U}_0\Vert_{\mathbf{\mathcal{B}}_{2,1}^{3/2}(\R^3)}+\Vert \mathbf{V}_0\Vert_ {\dot{B}_{2,\infty}^{-3/2}(\R^3)}$.  
\end{proposition}

\subsection{Proof of Proposition \ref{Proposition_Decay_Nonl}}
In this section, we prove \eqref{Estimate_M_Nonl}. This will be done through four  steps. 
\subsubsection*{Step 1. High-frequency estimate.} Keeping in mind that $\Delta_q f=\dot{\Delta}_q f,$ for $q\geq 0$, we show the estimates just for the inhomogeneous case. Recalling \eqref{High_Frequency_V} and keeping in mind \eqref{norm_H},  we have  

  \begin{equation}
\begin{aligned}
\Vert\Delta_q \Lambda^\ell [\mathcal{G}(t,x)\mathbf{U}_{0}]\Vert _{\mathcal{H}} \lesssim e^{-ct} \Vert\Delta_q \Lambda^\ell \mathbf{U}_{0}]\Vert _{\mathcal{H}}
\end{aligned}
\end{equation}
for all $q\geq 0$. Now, using Lemma \ref{Lemma_Dumhamel}, we have  
\begin{equation}\label{Duhamel_Prin_Main}
\begin{aligned}
\Vert\Delta_q\Lambda^{\ell}\mathbf{U}\Vert_{\mathcal{H}}\leq &\, \Vert\Delta_q \Lambda^\ell [\mathcal{G}(t,x)\mathbf{U}_{0}]\Vert _{\mathcal{H}}+\int_0^t\Vert \Delta_q \Lambda^\ell[\mathcal{G}(t-r,x)\mathcal{F}(\mathbf{%
U},\nabla \mathbf{U})(r)]\Vert_{\mathcal{H}} \textup{d}r\\
\lesssim  &\,  e^{-ct} \Vert\Delta_q \Lambda^\ell \mathbf{U}_{0}]\Vert _{\mathcal{H}}+\int_0^t e^{-c(t-r)} \Vert\Delta_q \Lambda^\ell \mathcal{F}(\mathbf{%
U},\nabla \mathbf{U})(r)]\Vert _{\mathcal{H}} \textup{d}r
\end{aligned}
\end{equation}
Keeping in mind \eqref{norm_H}, we write  
\begin{equation}\label{F_Norm}
\begin{aligned}
\Vert\Delta_q \Lambda^\ell \mathcal{F}(\mathbf{%
U},\nabla \mathbf{U})(r)]\Vert _{\mathcal{H}}
=&\, \Big\Vert\Delta_q \Lambda^\ell\Big(0,0,\dfrac{1}{\tau}\dfrac{B}{A}vw+\dfrac{2}{\tau}\nabla u\nabla v\Big)\Big\Vert_{\mathcal{H}}\\
\lesssim &\,\Vert\Delta_q \Lambda^\ell (vw+\nabla u\nabla v)\Vert_{L^2}. 
\end{aligned}
\end{equation}
This leads to by recalling \eqref{norm_H} and \eqref{Equiv_Besov_ell}, together with \eqref{F_Norm}  
\begin{equation}\label{Duhamel_H_F_Estimate}
\sum_{q\geq 0} 2^{q(3/2-\ell)}\Vert\Delta_q \Lambda^\ell\mathbf{V}\Vert_{L^2}\lesssim e^{-ct} \Vert \mathbf{V}_0\Vert_{B_{2,1}^{3/2}}+\int_0^t e^{-c(t-r)}\Vert (vw+\nabla u\nabla v)(r)\Vert_{\dot{B}_{2,1}^{3/2}}\textup{d}r
\end{equation}
for $0\leq \ell\leq 3/2$. Our goal now is to estimate the term $\Vert(vw+\nabla u\nabla v)(r)\Vert_{\dot{B}_{2,1}^{3/2}}$.  As $\dot{B}_{2,1}^{3/2}(\R^3)$ is an algebra, we have 
\begin{equation}\label{Nonl_Term_Decay}
\begin{aligned}
\Vert(vw+\nabla u\nabla v)\Vert_{\dot{B}_{2,1}^{3/2}}\lesssim&\, \Vert v\Vert_{\dot{B}_{2,1}^{3/2}}\Vert w\Vert_{\dot{B}_{2,1}^{3/2}}+\Vert \nabla u\Vert_{\dot{B}_{2,1}^{3/2}}\Vert \nabla v\Vert_{\dot{B}_{2,1}^{3/2}}\\
\lesssim&\, \Vert v\Vert_{\dot{B}_{2,1}^{3/2}}\Vert  w\Vert_{\dot{B}_{2,1}^{3/2}}+\Vert \nabla u\Vert_{\dot{B}_{2,1}^{3/2}}\Vert \Lambda^{\frac{3}{2}}\nabla v\Vert_{\dot{B}_{2,1}^{0 }}\\
\lesssim &\,(1+t)^{-\frac{3}{2}} \vecc{E}(t)\mathcal{M}(t)\\
\lesssim &\,(1+t)^{-\frac{3}{4}-\ell/2} \vecc{E}(t)\mathcal{M}(t).
\end{aligned}
\end{equation}
Hence, plugging \eqref{Nonl_Term_Decay} into \eqref{Duhamel_H_F_Estimate} and using the elementary inequality \begin{equation}\label{Exponential_Estimate}
\int_0^t(1+\sigma)^{-\beta} e^{-\gamma(t-\sigma)}\, \textup{d} \sigma\lesssim (1+t)^{-\beta},\qquad \gamma, \, \beta>0,    
\end{equation}
we obtain 
\begin{equation}\label{Duhamel_H_F_Estimate_2}
\sum_{q\geq 0} 2^{q(3/2-\ell)}\Vert\Delta_q \Lambda^\ell\mathbf{V}\Vert_{L^2}\lesssim e^{-ct} \Vert \mathbf{V}_0\Vert_{B_{2,1}^{3/2}}+(1+t)^{-\frac{\ell}{2}-\frac{3}{4}} \vecc{E}(t)\mathcal{M}(t)
\end{equation}
for $0\leq \ell\leq 3/2$.  We point out that we have $\Vert \mathbf{V}_0\Vert_{B_{2,1}^{3/2}} \lesssim  \Vert \mathbf{ U}_0\Vert_{\mathbf{\mathcal{B}}_{2,1}^{3/2}}$
\subsubsection*{Step 2. Low-frequency estimate.}
Now, our goal is to prove  the low-frequency estimate. We discuss separately  the two cases $0\leq \ell<3/2$ and $\ell=3/2$. 
\begin{description}
\item[(i) The case $0\leq \ell<3/2$] 
\end{description}
In this case and 
following \cite{Mori_Xu_Kawashima_2}, we have for the low-frequency estimate as in \eqref{Low_Frequency_Decay} and for $0\leq \ell<3/2$, 
\begin{equation}\label{Low_Freq_Main_1}
\Vert \Delta_{-1} \Lambda^\ell [\mathcal{G}(t,x)\mathbf{U}_{0}]\Vert_{\mathcal{H}}\lesssim \Vert \mathbf{V}_0\Vert_{\dot{B}_{2,\infty}^{-3/2}} (1+t)^{-\frac{\ell}{2}-\frac{3}{4}}.
\end{equation}
Hence, this yields by using Lemma \ref{Lemma_Dumhamel},  
\begin{equation}\label{Duh_Estimate_K}
\begin{aligned}
\Vert\Delta_{-1}\Lambda^{\ell}\mathbf{U}\Vert_{\mathcal{H}}\leq &\, \Vert\Delta_{-1} \Lambda^\ell [\mathcal{G}(t)\mathbf{U}_{0}]\Vert _{\mathcal{H}}+\int_0^t\Vert \Delta_{-1} \Lambda^\ell[\mathcal{G}(t-r)\mathcal{F}(\mathbf{%
U},\nabla \mathbf{U})(r)]\Vert_{\mathcal{H}} \textup{d}r\\
\lesssim  &\, \Vert \mathbf{V}_0\Vert_{\dot{B}_{2,\infty}^{-3/2}} (1+t)^{-\frac{\ell}{2}-\frac{3}{4}}
 +\mathrm{K}_1,
\end{aligned}
\end{equation}
where 
\begin{equation}
\mathrm{K}_1=\int_0^{t}\Vert \Delta_{-1} \Lambda^\ell[\mathcal{G}(t-r)\mathcal{F}(\mathbf{%
U},\nabla \mathbf{U})(r)]\Vert_{\mathcal{H}} \textup{d}r.
\end{equation}
To estimate $\mathrm{K}_1$,  applying \eqref{Low_Freq_Main_1} for $s=3/2$ and using the embedding $L^1(\R^3)\hookrightarrow \dot{B}_{2,\infty}^{-3/2}(\R^3) $, we deduce the following estimate 
\begin{equation}\label{K_1_Estimate}
\begin{aligned}
\mathrm{K}_1\lesssim&\, \int_0^{t} (1+t-r)^{-\frac{3}{4}-\frac{\ell}{2}}\Vert (vw+\nabla u\nabla v)(r)\Vert_{\dot{B}_{2,\infty}^{-3/2}}\textup{d}r\\
\lesssim&\,\int_0^{t} (1+t-r)^{-\frac{3}{4}-\frac{\ell}{2}}\Vert (vw+\nabla u\nabla v)(r)\Vert_{L^1}\textup{d}r.
\end{aligned}
\end{equation}
 We have by using H\"older's inequality,
\begin{equation}\label{Holder_1}
\Vert (vw+\nabla u\nabla v)(t)\Vert_{L^1}\leq \Vert \mathbf{V}(t)\Vert_{L^2}^2+\Vert
w(t)\Vert_{L^2}^2.
\end{equation}
Applying \eqref{L_p_Estimate}, we get $\Vert
w(t)\Vert_{L^2}\lesssim \Vert
w(t) \Vert_{\dot{B}^{0}_{2,1}}$ and 
hence, we obtain  
\begin{equation}\label{w_Estimate_Decay}
\Vert
w(t)\Vert_{L^2}^2\lesssim  (1+t)^{-3/2}  \mathcal{M}^2(t).
\end{equation}
Using the embedding $B_{2,1}^{3/2}\hookrightarrow L^2$, we find 
\begin{equation}\label{V_Estimate_Decay}
\Vert \mathbf{V}(t)\Vert_{L^2}^2\lesssim (1+t)^{-3/2}  \mathcal{M}^2(t). 
\end{equation}
Plugging  \eqref{w_Estimate_Decay} and  \eqref{V_Estimate_Decay} into \eqref{Holder_1}, we arrive at 
\begin{equation}\label{Holder_2}
\Vert (vw+\nabla u\nabla v)(t)\Vert_{L^1}\leq (1+t)^{-3/2}  \mathcal{M}^2(t).
\end{equation}
Hence, inserting  \eqref{Holder_2} into \eqref{K_1_Estimate}, we find 
\begin{equation}\label{K_1_Estimate_2}
\begin{aligned}
\mathrm{K}_1\lesssim&\, \mathcal{M}^2(t)
\int_0^{t} (1+t-r)^{-\frac{3}{4}-\frac{\ell}{2}}(1+r)^{-3/2}\textup{d}r\\
\lesssim&\, \mathcal{M}^2(t)(1+t)^{-\frac{3}{4}-\frac{\ell}{2}},
\end{aligned}
\end{equation}
where we applied   Lemma \ref{Lemma_Segel} in the last inequality. 

\noindent Now, plugging \eqref{K_1_Estimate_2} into \eqref{Duh_Estimate_K}, we deduce that  
\begin{equation}\label{Low_Frequency_Esti_0}
\Vert\Delta_{-1}\Lambda^{\ell}\mathbf{U}\Vert_{\mathcal{H}}\lesssim \Vert \mathbf{V}_0\Vert_{\dot{B}_{2,\infty}^{-3/2}} (1+t)^{-\frac{\ell}{2}-\frac{3}{4}}+ \mathcal{M}^2(t)(1+t)^{-\frac{3}{4}-\frac{\ell}{2}}.
\end{equation}
\begin{description}
\item[(ii) The case $\ell=3/2$] 
\end{description}

In this case, we have by using Lemma \ref{Lemma_Dumhamel},   
\begin{equation}\label{Duh_Estimate_3_2}
\begin{aligned}
\sum_{q<0}\Vert\dot{\Delta}_{q}\Lambda^{\frac{3}{2}}\mathbf{V}\Vert_{L^2} =&\, \sum_{q<0}\Vert\dot{\Delta}_{q}\Lambda^{\frac{3}{2}}\mathbf{U}\Vert_{\mathcal{H}}\\
\leq &\, \sum_{q<0}\Vert\dot{\Delta}_{q}\Lambda^{\frac{3}{2}} [\mathcal{G}(t)\mathbf{U}_{0}]\Vert _{\mathcal{H}}\\
&+\int_0^t \sum_{q<0}\Vert \dot{\Delta}_{q} \Lambda^{\frac{3}{2}}[\mathcal{G}(t-r)\mathcal{F}(\mathbf{%
U},\nabla \mathbf{U})(r)]\Vert_{\mathcal{H}} \textup{d}r\\
\lesssim  &\, \Vert \mathbf{V}_0\Vert_{\dot{B}_{2,\infty}^{-3/2}} (1+t)^{-\frac{3}{2}}
 +\mathrm{K}_2,
\end{aligned}
\end{equation}
where we have used \eqref{low_Frequncy_Estimate} together with the fact that  
$\Vert \Lambda^{\frac{3}{2}} \dot{\Delta}_{q}\mathbf{V}\Vert_{L^2}\approx 2^{q\frac{3}{2}}\Vert \dot{\Delta}_{q}\mathbf{V}\Vert_{L^2}$. To estimate $\mathrm{K}_2$, we apply \eqref{low_Frequncy_Estimate}, to get 
 \begin{equation}
 \begin{aligned}
\mathrm{K}_2=&\, \int_0^t \sum_{q<0}\Vert \dot{\Delta}_{q} \Lambda^{\frac{3}{2}}[\mathcal{G}(t-r)\mathcal{F}(\mathbf{%
U},\nabla \mathbf{U})(r)]\Vert_{\mathcal{H}} \textup{d}r\\
\lesssim &\,\int_0^t  (1+t-r)^{-\frac{3}{2}}\Vert (vw+\nabla u\nabla v)(r)\Vert_{\dot{B}_{2,\infty}^{-3/2}}\textup{d}r\\
\lesssim&\,\mathcal{M}^2(t)(1+t)^{-\frac{3}{2}}
\end{aligned}
\end{equation}
as in the estimate of $\mathrm{K}_1$. Consequently, we deduce that 
\begin{equation}\label{Low_estimate_3_2}
\sum_{q<0}\Vert\dot{\Delta}_{q}\Lambda^{\frac{3}{2}}\mathbf{V}\Vert_{L^2}\lesssim \Vert \mathbf{V}_0\Vert_{\dot{B}_{2,\infty}^{-3/2}} (1+t)^{-\frac{3}{2}}+\mathcal{M}^2(t)(1+t)^{-\frac{3}{2}}. 
\end{equation}
Collecting \eqref{Duhamel_H_F_Estimate_2} and \eqref{Low_estimate_3_2}, we obtain 
\begin{equation}\label{V_3_2_Estimate_0}
\begin{aligned}
\Vert\Lambda^{\frac{3}{2}}\mathbf{V}(r)\Vert_{\dot{B}^{0}_{2,1}}\lesssim&\,  \big(\mathbbmss{E}_0
+ \vecc{E}(t)\mathcal{M}(t)+\mathcal{M}^2(t)\big)(1+t)^{-\frac{3}{2}}. 
\end{aligned}
\end{equation}
\subsubsection*{Step 3. The estimate of $\Vert\mathbf{V}\Vert_{\dot{B}^{\frac{5}{2}}_{2,1}}$}

Taking $\ell=5/2$ in  \eqref{Duhamel_Prin_Main}, we get as in  Step 1:  
\begin{equation}\label{Duhamel_H_F_Estimate_V}
\sum_{q\geq 0} \Vert\dot{\Delta}_q \Lambda^{\frac{5}{2}}\mathbf{V}\Vert_{L^2}\lesssim e^{-ct} \Vert \mathbf{V}_0\Vert_{\dot{B}_{2,1}^{5/2}}+\int_0^t e^{-c(t-r)}\Vert (vw+\nabla u\nabla v)(r)\Vert_{\dot{B}_{2,1}^{5/2}}\textup{d}r. 
\end{equation}
Now, recalling \eqref{Prod_Estima_Hom}, we obtain   
\begin{equation}\label{Nonl_Term_Decay_5_2}
\begin{aligned}
\Vert(vw+\nabla u\nabla v)\Vert_{\dot{B}_{2,1}^{5/2}}\lesssim&\, \Vert w\Vert_{L^\infty} \Vert v\Vert_{\dot{B}_{2,1}^{5/2}}+\Vert v\Vert_{L^\infty} \Vert w\Vert_{\dot{B}_{2,1}^{5/2}}\\
&+\Vert \nabla u\Vert_{L^\infty}\Vert \nabla v\Vert_{\dot{B}_{2,1}^{5/2}}+\Vert \nabla v\Vert_{L^\infty}\Vert \nabla u\Vert_{\dot{B}_{2,1}^{5/2}}.
\end{aligned}
\end{equation}
Next, we have by using  the embedding $\dot{B}^{3/2}_{2,1}(\R^3) \hookrightarrow L^\infty(\R^3)$, together with \eqref{Equiv_Besov_ell}, 
\begin{equation}\label{First_Term_prod}
\begin{aligned}
&\Vert w\Vert_{L^\infty} \Vert v\Vert_{\dot{B}_{2,1}^{5/2}}+\Vert v\Vert_{L^\infty} \Vert w\Vert_{\dot{B}_{2,1}^{5/2}}\\
\lesssim&\, \Vert  w\Vert_{\dot{B}_{2,1}^{3/2}}\Vert  v\Vert_{\dot{B}_{2,1}^{5/2}}+\Vert  v\Vert_{\dot{B}_{2,1}^{3/2}} (\Vert  v\Vert_{\dot{B}_{2,1}^{5/2}}+\Vert v+\tau w\Vert_{\dot{B}_{2,1}^{5/2}})\\
\lesssim&\,\Vert  w\Vert_{\dot{B}_{2,1}^{3/2}}\Vert  \nabla v\Vert_{\dot{B}_{2,1}^{3/2}}+\Vert  v\Vert_{\dot{B}_{2,1}^{3/2}} (\Vert  \nabla v\Vert_{\dot{B}_{2,1}^{3/2}}+\Vert v+\tau w\Vert_{\dot{B}_{2,1}^{5/2}})\\
\lesssim&\,\Vert  w\Vert_{\dot{B}_{2,1}^{3/2}}\Vert  \mathbf{V}\Vert_{\dot{B}_{2,1}^{3/2}}+\Vert  v\Vert_{\dot{B}_{2,1}^{3/2}}(\Vert  \mathbf{V}\Vert_{\dot{B}_{2,1}^{3/2}}+\Vert\mathbf{V}\Vert_{\dot{B}^{5/2}_{2,1}}).
\end{aligned}
\end{equation}
Similarly, we have 
\begin{equation}\label{Second_Term_prod}
\begin{aligned}
&\Vert \nabla u\Vert_{L^\infty}\Vert \nabla v\Vert_{\dot{B}_{2,1}^{5/2}}+\Vert \nabla v\Vert_{L^\infty}\Vert \nabla u\Vert_{\dot{B}_{2,1}^{5/2}}\\
\lesssim&\,\Vert  \nabla u\Vert_{\dot{B}_{2,1}^{3/2}}\Vert\mathbf{V}\Vert_{\dot{B}^{\frac{5}{2}}_{2,1}}+\Vert  \nabla v\Vert_{\dot{B}_{2,1}^{3/2}}\Vert\mathbf{V}\Vert_{\dot{B}^{\frac{5}{2}}_{2,1}}. 
\end{aligned}
\end{equation}
Hence, collecting  \eqref{Nonl_Term_Decay_5_2}, \eqref{First_Term_prod} and \eqref{Second_Term_prod}, we get by plugging  into \eqref{Duhamel_H_F_Estimate_V} and recalling the definition of $\mathcal{M}(t)$,  
\begin{equation}\label{Duhamel_H_F_Estimate_V_2}
\begin{aligned}
\sum_{q\geq 0} 2^{5q/2}\Vert\dot{\Delta}_q \Lambda^{\frac{5}{2}}\mathbf{V}\Vert_{L^2}\lesssim&\, e^{-ct} \Vert \mathbf{V}_0\Vert_{\dot{B}_{2,1}^{5/2}}+ \vecc{E}(t)\mathcal{M}(t)\int_0^t e^{-c(t-r)}(1+r)^{-\frac{3}{2}}\textup{d}r\\
\lesssim&\, e^{-ct} \Vert \mathbf{V}_0\Vert_{\dot{B}_{2,1}^{5/2}}+ \vecc{E}(t)\mathcal{M}(t)(1+t)^{-\frac{3}{2}}, 
\end{aligned}
\end{equation}
 where we have used \eqref{Exponential_Estimate}.  We also point out that we have $\Vert \mathbf{V}_0\Vert_{\dot{B}_{2,1}^{5/2}}\lesssim   \Vert \mathbf{ U}_0\Vert_{\mathbf{\mathcal{B}}_{2,1}^{3/2}}$.

Now, for the low frequency estimate, we proceed as in \eqref{Duh_Estimate_3_2} and obtain 
\begin{equation}
\begin{aligned}
\sum_{q<0}\Vert\dot{\Delta}_{q}\Lambda^{\frac{5}{2}}\mathbf{V}\Vert_{L^2} \lesssim  &\, \Vert \mathbf{V}_0\Vert_{\dot{B}_{2,\infty}^{-3/2}} (1+t)^{-2}
 +\mathrm{K}_3
\end{aligned}
\end{equation}
with 
\begin{equation}
 \begin{aligned}
\mathrm{K}_3= &\,\int_0^t  (1+t-r)^{-2}\Vert (vw+\nabla u\nabla v)(r)\Vert_{\dot{B}_{2,\infty}^{-3/2}}\textup{d}r\\
\lesssim&\,\mathcal{M}^2(t)(1+t)^{-3/2}\\
\end{aligned}
\end{equation}
Hence, this yields 
\begin{equation}\label{Low_Freq_Estimate_5_2}
\sum_{q<0}\Vert\dot{\Delta}_{q}\Lambda^{\frac{5}{2}}\mathbf{V}\Vert_{L^2} \lesssim  \Vert \mathbf{V}_0\Vert_{\dot{B}_{2,\infty}^{-3/2}} (1+t)^{-2}
 +\mathcal{M}^2(t)(1+t)^{-3/2}. 
\end{equation}
Collecting \eqref{Duhamel_H_F_Estimate_V_2} and \eqref{Low_Freq_Estimate_5_2},  we deduce that 
\begin{equation}\label{Estimate_V_5_2_0}
\Vert\mathbf{V}(t)\Vert_{\dot{B}^{\frac{5}{2}}_{2,1}}\lesssim \big( \mathbbmss{E}_0+\vecc{E}(t)\mathcal{M}(t)+\mathcal{M}^2(t)\big)(1+t)^{-3/2}. 
\end{equation}
\subsubsection*{Step 4. The estimates of $\Vert\Lambda^{\frac{3}{2}}w\Vert_{\dot{B}^{0}_{2,1}}$ and $\Vert w(r)\Vert_{\dot{B}^{0}_{2,1}} $}

We multiply the third equation in \eqref{System_New_Localized} by $\dot{\Delta}_qw$ and integrate with respect to space,  using H\"older's inequality, to arrive at
\begin{equation}
\begin{aligned}
\frac{{1}}{2}\frac{\textup{d}}{\textup{d}t}\Vert \dot{\Delta}_qw\Vert_{L^2} ^{2}+\frac{1}{\tau}\Vert \dot{\Delta}_q w\Vert_{L^2} ^{2}\lesssim&\,\frac{1}{\tau} \Big(\Vert \Delta \dot{\Delta}_q u \Vert_{L^2}\Vert \dot{\Delta}_qw\Vert_{L^2} +\Vert \Delta \dot{\Delta}_q v \Vert_{L^2}\Vert \dot{\Delta}_qw\Vert_{L^2}\Big.\\
&\Big.+\Vert \mathrm{R}_q\Vert_{L^2}\Vert \dot{\Delta}_qw\Vert_{L^2} \Big)
\end{aligned}  
\end{equation}  
which yields 
\begin{equation}\label{w_Eenrgy_Nonl}
\begin{aligned}
\frac{\textup{d}}{\textup{d}t}\Vert \dot{\Delta}_qw\Vert_{L^2} +\frac{1}{\tau}\Vert \dot{\Delta}_q w\Vert_{L^2} \lesssim&\, \frac{1}{\tau}\Big(\Vert \Delta \dot{\Delta}_q u \Vert_{L^2} +\Vert \Delta \dot{\Delta}_q v \Vert_{L^2}+\Vert \mathrm{R}_q\Vert_{L^2}\Big) \\
\lesssim&\, \Big(\Vert \Delta \dot{\Delta}_q (u+\tau v)\Vert_{L^{2}}+\Vert \Delta \dot{\Delta}_q  v\Vert_{L^{2}}\Big.
\\
&\Big.+ \Vert \dot{\Delta}_q(  vw+\nabla u \cdot \nabla v)\Vert_{L^2}\Big). 
\end{aligned}  
\end{equation}
Then, we obtain  
\begin{equation}
\begin{aligned}
\frac{\textup{d}}{\textup{d}t}\Vert \dot{\Delta}_q w\Vert_{L^2}+\frac{1}{\tau}\Vert \dot{\Delta}_q w\Vert_{L^2} \lesssim\, \frac{1}{\tau}\Big(\Vert \nabla\dot{\Delta}_q\mathbf{V}\Vert_{L^2}+\Vert \dot{\Delta}_q(  vw+\nabla u \cdot \nabla v)\Vert_{L^2}\Big).
\end{aligned}   
\end{equation}
We  multiply the above inequality by $e^{ t}$, we arrive at
\begin{equation}
\begin{aligned}
\frac{\textup{d}}{\textup{d}t} \left(e^{\frac{1}{\tau} t}\Vert \dot{\Delta}_q w\Vert_{L^2} \right)\lesssim \frac{1}{\tau}e^{ \frac{1}{\tau}t}\left(\Vert \nabla\dot{\Delta}_q\mathbf{V}\Vert_{L^2}+\Vert \dot{\Delta}_q(  vw+\nabla u \cdot \nabla v)\Vert_{L^2}\right).     
\end{aligned}
\end{equation}
By  integrating with respect to time, we obtain 
\begin{equation}\label{Energy_First_Order}
\begin{aligned}
 \Vert \dot{\Delta}_q w\Vert_{L^2} \lesssim&\, \frac{1}{\tau}  e^{-
\frac{1}{\tau}t}\Vert \dot{\Delta}_q w_0 \Vert_{L^2} \\
 &+\frac{1}{\tau}\int_0^t e^{-\frac{1}{\tau}(t-r)}\left(\Vert \nabla\dot{\Delta}_q\mathbf{V}(r)\Vert_{L^2}
 +\Vert \dot{\Delta}_q(  vw+\nabla u \cdot \nabla v)(r)\Vert_{L^2}\right)\, \textup{d}r. 
\end{aligned}
\end{equation}
We multiply the above inequality by $2^{\frac{3q}{2}}$, summing the resulting inequality over $q\in \Z$, we get 
\begin{equation}
\begin{aligned}
\Vert w\Vert_{\dot{B}_{2,1}^{3/2}}\lesssim&\, \frac{1}{\tau}e^{-\frac{1}{\tau}t} \Vert w_0\Vert_{\dot{B}_{2,1}^{3/2}}+\int_0^te^{-\frac{1}{\tau}(t-r)}\Big(\Vert\mathbf{V}(r)\Vert_{\dot{B}_{2,1}^{5/2}}+\Vert (  vw+\nabla u \cdot \nabla v)(r)\Vert_{\dot{B}_{2,1}^{3/2}}\Big). 
\end{aligned}
\end{equation}

Now, recalling \eqref{Estimate_V_5_2_0},  \eqref{Nonl_Term_Decay} and using  \eqref{Exponential_Estimate}, we obtain 
\begin{equation}\label{w_Estimate_3_2}
\begin{aligned}
\Vert w\Vert_{\dot{B}_{2,1}^{3/2}}\lesssim&\, e^{-\frac{1}{\tau}t} \Vert w_0\Vert_{\dot{B}_{2,1}^{3/2}}
+\big( \mathbbmss{E}_0+\vecc{E}(t)\mathcal{M}(t)+\mathcal{M}^2(t)\big)
(1+t)^{-\frac{3}{2}}\\
\lesssim&\,  \big( \mathbbmss{E}_0+\vecc{E}(t)\mathcal{M}(t)+\mathcal{M}^2(t)\big)
(1+t)^{-\frac{3}{2}}.
\end{aligned}
\end{equation}
By the same method, we can further show that  summing \eqref{Energy_First_Order} over $q\in \Z$ (we omit the details) 
\begin{equation}
\begin{aligned}
\Vert w\Vert_{\dot{B}_{2,1}^{0}}\lesssim&\, e^{-\frac{1}{\tau} t} \Vert w_0\Vert_{\dot{B}_{2,1}^{0}}+\int_0^te^{-\frac{1}{\tau}(t-r)}\Big(\Vert\mathbf{V}(r)\Vert_{\dot{B}_{2,1}^{1}}+\Vert (  vw+\nabla u \cdot \nabla v)(r)\Vert_{\dot{B}_{2,1}^{0}}\Big). 
\end{aligned}
\end{equation}
First, we have from  \eqref{Duhamel_H_F_Estimate_2} and \eqref{Low_Frequency_Esti_0} (for $\ell=0$)
\begin{equation}
\Vert\mathbf{V}(r)\Vert_{\dot{B}_{2,1}^{1}}\lesssim \Vert\mathbf{V}(r)\Vert_{B_{2,1}^{3/2}}\lesssim \big( \mathbbmss{E}_0+\vecc{E}(t)\mathcal{M}(t)+\mathcal{M}^2(t)\big)
(1+r)^{-\frac{3}{4}}. 
\end{equation}

Now, using the estimate  (see Lemma \ref{Lemma_Danchin})
\begin{equation}
\begin{aligned}
\Vert (  vw+\nabla u \cdot \nabla v)(r)\Vert_{\dot{B}_{2,1}^{0}}\lesssim&\, \Vert v\Vert_{\dot{B}_{2,1}^{3/2}}\Vert w\Vert_{\dot{B}_{2,1}^{0}}+\Vert \nabla u\Vert_{\dot{B}_{2,1}^{3/2}}\Vert \nabla v\Vert_{\dot{B}_{2,1}^{0}}\\
\lesssim&\, \Vert v\Vert_{\dot{B}_{2,1}^{3/2}}\Vert w\Vert_{\dot{B}_{2,1}^{0}}+\Vert \nabla u\Vert_{\dot{B}_{2,1}^{3/2}}\Vert\mathbf{V}\Vert_{B_{2,1}^{3/2}}\\
\lesssim&\,\vecc{E}(t)\mathcal{M}(t)(1+r)^{-3/4}
\end{aligned}
\end{equation} 
Hence, we obtain as in the proof of \eqref{w_Estimate_3_2}, 

\begin{equation}\label{w_0_2_Estimate}
\Vert w\Vert_{\dot{B}^{0}_{2,1}}\lesssim  \big( \mathbbmss{E}_0+\vecc{E}(t)\mathcal{M}(t)+\mathcal{M}^2(t)\big)
(1+t)^{-\frac{3}{4}}. 
\end{equation}

\noindent Consequently, \eqref{Estimate_M_Nonl} holds by collecting \eqref{Duhamel_H_F_Estimate_2}, \eqref{Low_Frequency_Esti_0}, \eqref{V_3_2_Estimate_0}, \eqref{Estimate_V_5_2_0}, \eqref{w_Estimate_3_2} and \eqref{w_0_2_Estimate}. This finishes the proof of Proposition \ref{Proposition_Decay_Nonl}.  

\subsection{Proof of Theorem \ref{Decay_Estimate_Theorem}}\label{Sec:Proof_Theorem_2}
From Theorem \ref{Main_Theorem_Nonl}, we have 
\begin{equation}
\vecc{E}(t)\lesssim \Vert \mathbf{ U}_0\Vert_{\mathbf{\mathcal{B}}_{2,1}^{3/2}}\lesssim  \mathbbmss{E}_0. 
\end{equation}
Hence, for $\mathbbmss{E}_0$ sufficiently small, we have  
from \eqref{Estimate_M_Nonl} 
\begin{equation}\label{Estimate_M_Nonl_Main}
\mathcal{M}(t)\lesssim  \mathbbmss{E}_0+\mathcal{M}^2(t).
\end{equation} 
Applying Lemma \ref{Lemma_Stauss}, we deduce from \eqref{Estimate_M_Nonl_Main} that 
\begin{equation}\label{Estimate_final_Decay}
\mathcal{M}(t)\lesssim  \mathbbmss{E}_0
\end{equation}
provided that $ \mathbbmss{E}_0$ is sufficiently small. 
Consequently, the decay estimates in Theorem \ref{Decay_Estimate_Theorem} results from \eqref{Estimate_final_Decay}. This ends the proof of Theorem \ref{Decay_Estimate_Theorem}. 

\appendix
\begin{appendices}  
\section{Analytic tools}\label{Appendix_A}
In this appendix, we introduce the Littlewood--Paley decomposition,  define the Besov spaces and list some of their properties. See e.g. \cite{bahouri2011fourier} for more details. 
\subsection{Littlewood-Paley theory}
 Let $\mathcal{S}(\mathbb{R}^n)$ be the Schwartz class of rapidly decreasing functions. For a given $f$
in $\mathcal{S}(\mathbb{R}^n)$, we define its Fourier transform $\mathcal{F}%
f=\hat{f}$ and its inverse $\mathcal{F}^{-1}\hat{f}$ as 
\begin{equation}
\hat{f}(\xi)=\int_{\mathbb{R}^n}e^{-ix\cdot \xi} f(x)
\dx,\quad \text{and}\quad \mathcal{F}^{-1}\hat{f}(x)=\frac{1}{(2\pi)^{n}}\int_{\mathbb{R}^n}e^{i x\cdot\xi}\hat{%
f}(\xi)\textup{d}\xi.
\end{equation}

Let $(\varphi, \chi(\xi))$ be a couple of smooth functions  valued in $[0,1]$ such that $\varphi$ is supported in
the annulus $\mathcal{C}=\{\xi\in \mathbb{R}^n,\,\, \frac{3}{4}\leq |\xi|\leq \frac{8}{3}%
\}$, $\chi$ is supported in the ball 
 $\mathcal{B}_0=\{\xi\in \mathbb{R}^n,\,\,|\xi|\leq \frac{4}{3}\}$, such that  
\begin{equation}
\chi(\xi)+\sum_{q\geq 0} \varphi(2^{-q}\xi)=1,\quad \xi\in \mathbb{R}^n
\end{equation}
and 
\begin{equation}
\sum_{q\in\mathbb{Z}} \varphi(2^{-q}\xi)=1,\quad \xi\in \mathbb{R}^n \setminus
\{0\}.
\end{equation}
Denote by  $h:=\mathcal{F}^{-1} \varphi$ and $\widetilde{h}:=\mathcal{F}^{-1} \chi$.  For $f\in \mathcal{S}'$, define 
\begin{equation}
\begin{aligned}
\dot{\Delta}_q f=&\, \varphi(2^{-q}D )  f=2^{dq}\int_{\mathbb{R}^n}h(2^qy)f(x-y)\textup{d}y,\quad q\in \Z;\\
\Delta_q f=&\, \varphi(2^{-q}D )  f=2^{dq}\int_{\mathbb{R}^n}h(2^qy)f(x-y)\textup{d}y,\quad q\geq 0;\\
\Delta_{-1} f=&\, \chi(D )  f=\int_{\mathbb{R}^n}\widetilde{h}(y)f(x-y)\textup{d}y,\qquad \Delta_{q} f=0\,\ \text{for}\,\ q\leq -2.
\end{aligned}
\end{equation}
Observe that $\dot{\Delta}_q$ coincide with $\Delta_q$ for $q\geq 0$.

%
Now, denote by $\mathcal{S}_0^\prime=\mathcal{S}^\prime/\mathcal{P}$ the
tempered distribution modulo polynomials $\mathcal{P}$.
It is not hard to see that the space $\mathcal{S}_0^\prime$ is exactly the
space of tempered distributions for which we may write 
\begin{equation}
u=\sum_{q\in \mathbb{Z}}\dot{\Delta}_qu.
\end{equation}
This decomposition is called the homogeneous Littlewood-Paley decomposition.

We recall 
the Bernstein inequality, which is the fundamental inequality in the  Littlewood--Paley theory. See \cite{bahouri2011fourier}.  

\begin{lemma}[Bernstein inequality]
\label{Lemma_Bernstein}
Let $0<R_1<R_2$ and $	1\leq a\leq b\leq \infty$ and  
let $k$ be an integer. 
\begin{description}
\item[ (i)] If $\supp (\hat{f})\subset\{\xi\in \R^n: |\xi|\leq R_1 \lambda\}$, then 
\begin{equation}
\Vert \Lambda ^kf\Vert_{L^a}\lesssim \lambda^{k+n\big(\frac{1}{a}-\frac{1}{b}\big)}\Vert f\Vert_{L^b}. 
\end{equation}
\item[ (ii)] If $\supp (\hat{f})\subset\{\xi\in \R^n: R_1 \lambda\lesssim |\xi|\leq R_2 \lambda\}$, then 
\begin{equation}\label{Berns_2}
\Vert \Lambda^kf\Vert_{L^a}\approx \lambda^{k}\Vert f\Vert_{L^a}. 
\end{equation}
\end{description}
As a consequence of the above inequality, we have 
\begin{equation}\label{Equiv_Besov_ell}
\Vert \Lambda^\ell f\Vert_{B_{p,r}^s}\lesssim \Vert  f\Vert_{B_{p,r}^{s+\ell}}\,\quad (\ell\geq 0);\qquad \Vert \Lambda^\ell f\Vert_{\dot{B}_{p,r}^s}\approx \Vert f\Vert_{\dot{B}_{p,r}^{s+\ell}}\quad (\ell\in \R). 
\end{equation}

\end{lemma}

\subsection{Besov spaces}

Now, we introduce the definition of Besov spaces. These spaces are very
useful in the study of nonlinear PDEs since many embedding theorems which fail in Sobolev spaces are correct  in Besov spaces.  

\begin{definition}
For $s\in \mathbb{R}$ and $1\leq p,r\leq \infty$, the homogeneous Besov space  
$\dot{B}_{p,r}^{s}$ is defined as 
\begin{equation}
\dot{B}_{p,r}^{s}=\left\{f\in \mathcal{S}_0^\prime : \Vert f\Vert_{\dot{B}%
_{p,r}^{s}}<\infty\right\},
\end{equation}
where 
\begin{equation}
\Vert f\Vert_{\dot{B}_{p,r}^{s}}=\left\{ 
\begin{array}{ll}
\bigg(\displaystyle\sum_{q\in \mathbb{Z}} \left(2^{qs}\Vert \dot{\Delta}_q
f\Vert_{L^p}\right)^r\bigg)^{1/r}, & \qquad r<\infty\vspace{0.2cm} \\ 
\displaystyle\sup_{j\in\mathbb{Z}}2^{qs}\Vert \dot{\Delta}_q f\Vert_{L^p}, & 
\qquad r=\infty. 
\end{array}
\right.
\end{equation}
\end{definition}

Similarly, we define the inhomogeneous Besov space $B^{s}_{p,r}$ as:

\begin{definition}
For $s\in \mathbb{R}$ and $1\leq p,r\leq \infty$, the inhomogeneous Besov
space $B^{s}_{p,r}$ is defined as 
\begin{equation}
B^{s}_{p,r}=\left\{f\in \mathcal{S}^\prime : \Vert
f\Vert_{B^{s}_{p,r}}<\infty\right\},
\end{equation}
where 
\begin{equation}
\Vert f\Vert_{B_{p,r}^{s}}=\left\{ 
\begin{array}{ll}
\bigg(\displaystyle\sum_{q= -1}^\infty \left(2^{qs}\Vert \Delta_q
f\Vert_{L^p}\right)^r\bigg)^{1/r}, & \qquad r<\infty\vspace{%
0.2cm} \\ 
\displaystyle\sup_{q\geq -1}2^{qs}\Vert \Delta_q f\Vert_p,
& \qquad r=\infty.%
\end{array}
\right.
\end{equation}
\end{definition}

We summarize the main properties of Besov spaces in the following lemma. See \cite[Lemma 2.2]{Kawashima_1} for the properties (1)-(5) and \cite[Proposition 1.4]{Danchin_2001} for the properties (6) and (7). 
\begin{lemma}
Let $s\in \R$ and $1\leq p,r\leq \infty$, then 
\begin{enumerate}[label=\arabic*.,ref=.\arabic*]
\item If $s>0$, then $B^{s}_{p,r}=L^p\cap \dot{B}_{p,r}^{s}$ and $\Vert
f\Vert_{B_{p,r}^{s}}\approx \Vert f\Vert_{p}+\Vert f\Vert_{\dot{B}%
_{p,r}^{s}} $.
\item If $\tilde{s}\leq s$, then
$B^{s}_{p,r}\hookrightarrow B^{\tilde{s}}_{p,r} $. 
\item If $1\leq r\leq \tilde{r}\leq \infty$, then $\dot{B}_{p,r}^{s}\hookrightarrow\dot{B}_{p,\tilde{r}}^{s}$ and $B_{p,r}^{s}\hookrightarrow B_{p,\tilde{r}}^{s}$.
\item If $1\leq p\leq \tilde{p}\leq \infty$, then $\dot{B}_{p,r}^{s}\hookrightarrow\dot{B}_{\tilde{p},r}^{s-n(\frac{1}{p}-\frac{1}{\tilde{p}})}$ and $B_{p,r}^{s}\hookrightarrow B_{\tilde{p},r}^{s-n(\frac{1}{p}-\frac{1}{\tilde{p}})}$.
\item $\dot{B}_{p,1}^{n/p}\hookrightarrow \mathcal{C}_0$ and $B_{p,1}^{n/p}\hookrightarrow \mathcal{C}_0$\quad $(1\leq p<\infty)$, where $\mathcal{C}_0$ is the space of bounded continuous functions which decay at infinity. 
\item There exists a universal constant $C$ such that 
\begin{equation}\label{Eqv_Norms}
C^{-1}\Vert f\Vert_{\dot{B}_{p,r}^s}\leq \Vert \nabla f\Vert_{\dot{B}_{p,r}^{s-1}}\leq C\Vert f\Vert_{\dot{B}_{p,r}^s}. 
\end{equation}
\item $\Vert \nabla u\Vert_{B_{p,r}^{s-1}}\lesssim \Vert  u\Vert_{B_{p,r}^{s}}$. 

\end{enumerate}

\end{lemma}

The Littlewood--Paley decomposition enables us to obtain bounds for each
dyadic block in spaces of the form $L^\rho_T(L^p)$. Therefore, to go from
this type of bounds to estimates in $L^\rho_T(\dot{B}_{p,r}^s)$ we need to
perform a summation in $\ell^r(\mathbb{Z})$. By doing this, we in
fact do not bound the $L^\rho_T(\dot{B}_{p,r}^s)$ norm, since the time
integration has been performed before the summation in $\ell^r(\mathbb{Z})$.
This requires the introduction of the following spaces known as the
Chemin--Lerner spaces.  See \cite{Chemin_Lerner_1995}. 

\begin{definition}
Let $ s\in \mathbb{R},\, 1\leq p,\, r,\,\theta\leq \infty$ and $I\subset 
\mathbb{R}$ be  an interval. Then the homogeneous mixed time-space
Chemin--Lerner space $\widetilde{L}^\theta_T(\dot{B} _{p,r}^{s})$ is defined by 
\begin{equation}
\widetilde{L}^\theta_T(\dot{B} _{p,r}^{s})=\left\{f\in L^\theta(0,T; \mathcal{S}%
^\prime): \Vert f\Vert_{\widetilde{L}^\theta_T(\dot{B} _{p,r}^{s})}<\infty\right\},
\end{equation}
where 
\begin{equation}
\Vert f\Vert_{\widetilde{L}^\theta_T({\dot{B}} _{p,r}^{s})}=\left\Vert 2^{sq}\left(\int_0^T
\Vert \Delta_q f(\tau)\Vert_p^\theta d\tau\right)^{1/\theta}\right\Vert_{\ell^r(%
\mathbb{Z})}
\end{equation}
with the usual modification if $\theta=\infty$. 
Notice here that the
integration in time is taken before the summation in $\ell^q(\mathbb{Z})$.

We also need the inhomogeneous mixed time-space $\widetilde{L}^\theta_T(B
_{p,q}^{s}), s>0$ whose norm is defined by 
\begin{equation}
\Vert f\Vert_{\widetilde{L}^\theta_T(B _{p,q}^{s})}=\Vert f\Vert_{L^\theta_T(L^p)}+\Vert
f\Vert_{\widetilde{L}^\theta_T(\dot{B} _{p,q}^{s})}.
\end{equation}

\end{definition}
We further define 
\begin{subequations}
\begin{equation}\label{C_1_T_Def}
\begin{aligned}
\widetilde{\mathcal{C}}_T (B_{p,q}^{s}):=&\,\widetilde{L}^\infty_T(B_{p,q}^{s})\cap \mathcal{C}([0,T]; B_{p,q}^{s})\\
\widetilde{\mathcal{C}}_T^1 (B_{p,q}^{s}):=&\,\{f\in \mathcal{C}^1 ([0,T]; B_{p,q}^{s}): \partial_t f\in\widetilde{L}^\infty_T(B_{p,q}^{s}) \},
\end{aligned}
\end{equation}
and for $T=+\infty$, we define 
\begin{equation}\label{C_1_infty_Def}
\begin{aligned}
\widetilde{\mathcal{C}} (B_{p,q}^{s}):=&\,\{f\in \mathcal{C} (\R+; B_{p,q}^{s}): \Vert f\Vert_{\widetilde{L}^\infty(B_{p,q}^{s})}<+\infty\}\\
\widetilde{\mathcal{C}}^1 (B_{p,q}^{s}):=&\,\{f\in \mathcal{C}^1 (\R_+; B_{p,q}^{s}): \Vert \partial_tf\Vert_{\widetilde{L}^\infty(B_{p,q}^{s})} \}
\end{aligned}
\end{equation}
where for notation simplicity,  the index $T$ will be omitted when $T=+\infty$.

\end{subequations} 
\begin{remark}
Using Minkowski's inequality, we may easily show  that the above spaces
can be linked to the usual space-time mixed spaces $L^\theta_T(\dot{B}
_{p,r}^{s}) $ through the embedding inequalities 
\begin{equation}  \label{Mink}
\left.
\begin{array}{ll}
\Vert f\Vert_{\widetilde{L}^\theta_T(\dot{B} _{p,r}^{s})}\leq \Vert f\Vert_{L^\theta_T(%
\dot{B} _{p,r}^{s})},\qquad \text{if}\quad r\geq \theta,
\vspace{0.3cm}\\
 \Vert f\Vert_{\widetilde{%
L}^\theta_T(\dot{B} _{p,r}^{s})}\geq \Vert f\Vert_{L^\theta_T(\dot{B}
_{p,r}^{s})},\qquad \text{if}\quad r\leq \theta.
\end{array}
\right. 
\end{equation}


In addition the following product estimate holds  (see \cite[Lemma 2.4]{Kawashima_1}) 
\begin{equation}
\Vert fg\Vert_{\widetilde{L}^\theta_T(\dot{B} _{p,r}^{s})}\lesssim \Vert f\Vert_{%
L^{\theta_1}_T(L^\infty)}\Vert g\Vert_{\widetilde{L}^{\theta_2}_T(\dot{B}
_{p,r}^{s})}+\Vert g\Vert_{L^{\theta_3}_T(L^\infty)}\Vert f\Vert_{\widetilde{L%
}^{\theta_4}_T(\dot{B} _{p,r}^{s})}
\end{equation}
whenever $s>0,\, 1\leq p\leq \infty$, $1\leq \theta, \theta_1, \theta_2,\theta_3,\theta_4\leq \infty$
and 
\begin{equation}
\frac{1}{\theta}=\frac{1}{\theta_1}+\frac{1}{\theta_2}=\frac{1}{\theta_3}+\frac{1}{\theta_4}.
\end{equation}
As a direct corollary, one has 
\begin{equation}
\Vert fg\Vert_{\widetilde{L}^\theta_T(\dot{B} _{p,r}^{s})}\lesssim \Vert g\Vert_{%
\widetilde{L}^{\theta_1}_T(\dot{B} _{p,r}^{s})}\Vert f\Vert_{\widetilde{L}^{\theta_2}_T(\dot{B}
_{p,r}^{s})}
\end{equation}
whenever $sp\geq n$ and $1/\theta=1/\theta_1+1/\theta_2$.
\end{remark}

\subsection{Useful  inequalities } 
Now, we introduce some useful inequalities such as the Moser-type product estimates, commutator estimates and the   Gagliardo--Nirenberg type inequalities and some other useful inequalities. 

 We start with the following characterization of the homogeneous Besov spaces. (See \cite[Corollary 1.3.2]{G_Gui_2013})
\begin{lemma}
Let  $1\leq p,r\leq \infty$ and $u\in \mathcal{S}^\prime(\R^n)$. Then $u$ belongs to $\dot{B}_{p,r}^s(\R^n)$ if and only if there exists a sequence $(c_q)_{q\in \Z}$ such that $c_q> 0$,  $\Vert c_q\Vert_{\ell^r}=1$ and  
\begin{equation}\label{Besov_Dya_Ineq}
\Vert \dot{\Delta}_q u\Vert_{L^p}\lesssim  c_q2^{-sq}\Vert u\Vert_{\dot{B}_{p,r}^s}. 
\end{equation}
\end{lemma}
 Next, we recall the following lemma. (see \cite[Proposition 2.1]{Kawashima_1}) 
\begin{lemma}[Product estimate]
Let $s>0$ and $1\leq p,r\leq \infty$. Let $f$ and $g$ be in $L^\infty\cap \dot{B}^s_{p,r}$. Then, it holds that 
\begin{equation}\label{Prod_Estima_Hom}
\Vert fg\Vert_{\dot{B}^s_{p,r}}\lesssim \Vert f\Vert_{L^\infty}\Vert g\Vert_{\dot{B}^s_{p,r}}+\Vert g\Vert_{L^\infty}\Vert f\Vert_{\dot{B}^s_{p,r}},
\end{equation}
where the hidden constant is depending only on $n,\, p$ and $s$. 

\noindent Since $\dot{B}^{n/p}_{p,1}$ is embedded in $L^\infty$, then it is an algebra and it holds that 
\begin{equation}
\Vert fg\Vert_{\dot{B}^{n/p}_{p,1}}\lesssim \Vert f\Vert_{\dot{B}^{n/p}_{p,1}}\Vert g\Vert_{\dot{B}^{n/p}_{p,1}}.
\end{equation}
\end{lemma}

\begin{lemma}[\cite{Danchin_2007}]\label{Lemma_Danchin}
The following inequality holds true:
\begin{equation}
\Vert fg\Vert_{\dot{B}^{s}_{2,1}}\lesssim \Vert f\Vert_{\dot{B}^{n/2}_{2,1}}\Vert g\Vert_{\dot{B}^{s}_{2,1}}
\end{equation}
whenever $s\in (-n/2, n/2]$. 
\end{lemma}

Now, we state some  helpful commutator estimates (see \cite[Proposition 2.3]{Mori_Xu_Kawashima_2}).
\begin{lemma}[Commutator estimates]
Let $1<p<\infty$, $1\leq \theta\leq \infty$ and $s\in (-\frac{n}{p}-1,\frac{n}{p}]$. Then there exists a generic constant $C$ depending only on $s$ and $n$ such that 
\begin{equation}\label{Commu_A_1}
\big\Vert [f,\dot{\Delta}_q]g\big\Vert_{L^p}\leq Cc_q 2^{-q(s+1)}\Vert f\Vert_{\dot{B}^{n/p+1}_{p,1}}\Vert g\Vert_{\dot{B}^{s}_{p,1}}
\end{equation}
and 
\begin{equation}\label{Commutator_2}
\big\Vert [f,\dot{\Delta}_q]g\big\Vert_{L_T^\theta L^p}\leq Cc_q 2^{-q(s+1)}\Vert f\Vert_{\widetilde{L}_T^{\theta_1}(\dot{B}^{n/p+1}_{p,1})}\Vert g\Vert_{\widetilde{L}_T^{\theta_2}(\dot{B}^{s}_{p,1})}
\end{equation}  
with $1/\theta=1/\theta_1+1/\theta_2$ and $(c_q)_{q\in \Z}$ denotes a positive sequence such that $\Vert c_q\Vert_{\ell^1}\leq1. $
\end{lemma}

\begin{lemma}
\label{Embedding_Lemma} Suppose that $\varrho>0$ and $1\leq p<2$. It holds
that 
\begin{equation}
\Vert f\Vert_{\dot{B}^{-\varrho}_{r,\infty}}\lesssim\Vert f\Vert_{L^p}
\end{equation}
with $1/p-1/r=\varrho/n$. In particular this holds with $\varrho=n/2,\, r=2$
and $p=1$.
\end{lemma}

We also have the following interpolation inequality 
see for instance \cite[Lemma 8.2]{Xu_Kawashima_2015}.   
\begin{lemma}
Suppose that $k\geq 0$ and $m,\, \varrho>0$. Then the following inequality holds 
\begin{equation}\label{Interpolation_I}
\Vert \Lambda^k f\Vert_{L^2}\lesssim \Vert \Lambda^{k+m}f\Vert_{L^2}^{\theta} \Vert f\Vert_{\dot{B}_{2,\infty}^{-\varrho}}^{1-\theta} \quad \text{with}\quad \theta=\frac{\varrho+k}{\varrho+k+m}. 
\end{equation}
\end{lemma}

We recall now the following inequality which has been proved in \cite{Se68}.   
\begin{lemma} 
\label{Lemma_Segel}
\label{Integral_lemma} Let $a>0$ and $\,b>0$ be constants such that $\max (a,b)>1.$ Then, we have 
\begin{equation}
\int_{0}^{t}\left( 1+t-r\right) ^{-a}\left( 1+r\right) ^{-b}\textup{d}r\leq C\left(
1+t\right) ^{-\min \left( a,b\right) }. \label{First_integral_inequality}
\end{equation}
\end{lemma}
The next lemma has been proved in \cite[Lemma 3.7]{Strauss_1968}.

\begin{lemma}
\label{Lemma_Stauss} Let $\mathrm{F}=\mathrm{F}(t)$ be a non-negative continuous function
satisfying the inequality
\begin{equation}
\mathrm{F}(t)\leq c_1+c_2 \mathrm{F}(t)^{\kappa},
\end{equation}
in some interval containing $0$, where $c_1$ and $c_2$ are positive
constants and $\kappa>1$. If $\mathrm{F}(0)\leq c_1$ and
\begin{equation}
c_1c_2^{1/(\kappa-1)}<(1-1/\kappa)\kappa^{-1/(\kappa-1)},
\end{equation}
then in the same interval
\begin{equation}
\mathrm{F}(t)<\frac{c_1}{1-1/\kappa}.
\end{equation}
\end{lemma}

\end{appendices} 

\section*{Acknowledgements}
The author would like to warmly thank Vanja Nikoli\'c and Reinhard Racke for sharing their  many insights about this problem and for the many discussions on the  first draft of this work. The author would also like to thank the reviewers for the careful reading of the manuscript and the thoughtful comments, which greatly helped improve this work.


\end{document}